\documentclass[journal]{IEEEtran}

\ifCLASSINFOpdf
\else
\fi

\usepackage{amsmath}
\usepackage{amssymb}
\usepackage{graphicx}
\usepackage{color}
\usepackage{dsfont,latexsym}
\usepackage{tikz}

\usepackage{enumitem}

\newtheorem{lemma}{Lemma}
\newtheorem{remark}{Remark}
\newtheorem{theorem}{Theorem}
\newtheorem{property}{Property}
\newtheorem{proof}{Proof}

\newtheorem{assumption}{Assumption}

\allowdisplaybreaks
\newcommand{\modif}[1]{{\textcolor{black}{#1}}}
\begin{document}
\title{Output-feedback stabilization of an underactuated network of $N$ interconnected $n+ m$ hyperbolic PDE systems}

\author{Jean Auriol% <-this % stops a space
\thanks{J. Auriol is with Université Paris-Saclay, CNRS, CentraleSupélec, Laboratoire des Signaux et Systèmes, 91190 Gif-sur-Yvette, France, jean.auriol@centralesupelec.fr. This project received funding from the Agence Nationale de la Recherche via grant PANOPLY ANR-23-CE48-0001-01.}% <-this % stops a space
}


\markboth{}%
{S}



\maketitle

\begin{abstract}

In this article, we detail the design of an output feedback stabilizing control law for an underactuated network of $N$ subsystems of $n+m$ heterodirectional linear first-order hyperbolic Partial Differential Equations interconnected through their boundaries. The network has a chain structure, as only one of the subsystems is actuated. The available measurements are located at the opposite extremity of the chain. The proposed approach introduces a new type of integral transformation to tackle in-domain couplings in the different subsystems while guaranteeing a ‘‘clear actuation path" between the control input and the different subsystems. Then, it is possible to state several essential properties of each subsystem: output trajectory tracking, input-to-state stability, and predictability (the possibility of designing a state prediction). We recursively design a stabilizing state-feedback controller by combining these properties. We then design a state-observer that reconstructs delayed values of the states. This observer is combined with the state-feedback control law to obtain an output-feedback controller. Simulations complete the presentation.
\end{abstract}


% \fi
\IEEEpeerreviewmaketitle

%This class of system may appear when considering oil production systems made of networks of pipes (whose principal line is known as the manifold)~\cite{mokhtari2017performance} or traffic network systems~\cite{yu2023traffic,yu2022simultaneous}.

\section{Introduction}
\IEEEPARstart{T}{he} interconnection of hyperbolic systems (potentially coupled with ODEs) represents a well-established topic, given its inherent occurrence in various industrial contexts (e.g., electric power transmission systems~\cite{schmuck2014feed}, control of after-treatment devices in exhaust lines~\cite{depcik2005one}, or traffic networks~\cite{yu2023traffic}). Specifically, interconnections characterized by a \textbf{cascade chain structure} have garnered notable attention~\cite{aarsnes2018torsional}. \modif{This particular network configuration holds significance due to its capacity to model intricate industrial phenomena, such as the propagation of torsional waves in drilling systems~\cite{aarsnes2018torsional}, deepwater construction vessels~\cite{stensgaard2010subsea}, density-flow systems as lossless electrical lines, frictionless open channels, or gas pipes~\cite{bastin2016stability,bastin2013exponential}.} %In this paper, we design an output-feedback controller to stabilize an arbitrary number of non-scalar PDE subsystems interconnected through their boundaries with a chain structure, the actuation being located at one extremity of the chain.

Most existing constructive control strategies for interconnected systems are grounded in the backstepping approach. Notably, significant attention has been directed towards cascaded interconnections of hyperbolic PDE-ODE systems, as evidenced in~\cite{aamo2012disturbance,deutscher2021backstepping,auriol2018delay}, as well as ODE-PDE-ODE configurations~\cite{deutscher2018output,BouSaba2019,wang2020delay}, where the design of control hinges upon the reformulation of the interconnection as a \emph{time-delay system}. Recent advancements have also emerged for interconnected PDE systems containing non-linear ODEs by designing a modular approach involving tracking controllers~\cite{irscheid2021observer,irscheid2023output}. These contributions generally establish stringent rank conditions on the various coupling matrices, among other requisites. In numerous scenarios encompassing underactuated PDEs (such as the simple interconnection of two scalar hyperbolic systems~\cite{auriol2019delay}, wherein only one of the subsystems is actuated), these conditions remain unmet, despite the existence of stabilizing controllers~\cite{auriol2019delay}.  This explains why underactuated PDEs have been the source of several contributions these last few years.

While the design of comprehensive control strategies for all types of underactuated systems or network configurations appears to be currently overly ambitious, several existing methodologies in the literature have put forward constructive control designs specifically tailored to networks with well-defined \textbf{structural characteristics}. Particularly noteworthy are \textbf{chain} configurations featuring a cascade structure, which have garnered significant attention. In such configurations, the network is a straight line, and the actuator/sensor is located at one of its extremities. This class of systems can arise in scenarios like oil production systems comprised of interconnected pipes, where the main conduit is known as the manifold~\cite{mokhtari2017performance}. More precisely, the lower part of the drill string is usually made up of drill collars that can significantly impact global dynamics due to their inertia. These pipes may have distinct lengths, densities, inertia, or Young's modulus. The variations in characteristic line impedance across space may lead to reflections manifesting at junctions. Such simple chain-structured networks can also model ventilation within buildings~\cite{witrant2010wireless}, density-flow systems~\cite{hayat2019pi}, open canals~\cite{de2003boundary}, or traffic systems, as described in~\cite{yu2023traffic,yu2022simultaneous} in the case of two cascaded freeway segments. \modif{Among other examples of interest, we can cite networks of 1-D flexible multi-structures~\cite{dager2006wave} as interconnected Timoshenko beams that can be used to model compliant mechanical structures such as cantilevers or flexible endoscopes.}

Recently, several methodologies have emerged to formulate stabilizing controllers for such chain structures. Among these methodologies, PI boundary controllers have been explored in~\cite{bastin2016stability,hayat2019pi} for fully actuated networks (i.e., networks with one control input per set of heterodirectional PDEs). Explicit stability conditions were obtained by utilizing suitable quadratic Lyapunov functions. In~\cite{knuppel2014control}, the authors consider a flatness-based design of a feedforward control of tree-like transmission networks. Analogous scenarios involving interconnected systems have been scrutinized in~\cite{su2017boundary}, where a velocity recirculation phenomenon in a wave equation was considered. Furthermore, the assessment of exact boundary controllability for nodal profiles of quasilinear hyperbolic systems with interface conditions in tree-like networks was conducted in~\cite{wang2014exact} using the method of characteristics. In a more recent study, detailed in \cite{strecker2017output}, the authors investigated output feedback control strategies for interconnections of $2 \times 2$ semilinear hyperbolic systems using the dynamics on the characteristic lines. Backstepping-based controllers have been developed in~\cite{deutscher2021backstepping,auriol2019delay,auriol2020output} in the case of interconnected scalar subsystems. While these advancements have contributed significantly to the field, they exhibited adaptability limitations across different chain structures. For example, integrating an additional PDE subsystem within the chain structure was not feasible in~\cite{auriol2020output}. To overcome this limitation, a novel approach was proposed in~\cite{redaud2021SCL} through the introduction of a \textbf{new recursive dynamics framework}. This framework, which is grounded in innovative prediction-based control laws~\cite{artstein1982linear,bekiaris2014simultaneous,bresch2017prediction} for difference equations, exhibits \textbf{modularity}. This modularity stems from the requirement that the control law only necessitates fundamental properties for each subsystem (controllability, trackability, observability, predictability). However, it should be noted that the techniques proposed in~\cite{redaud2021SCL} do not lend themselves directly to non-scalar subsystems. In the case of two non-scalar subsystems, a two-step procedure was proposed in~\cite{auriol2022robust}.%, adding a flatness-based feedforward tracking component~\cite{meurer2009tracking,hu2015boundary} to the design. %Nevertheless, this approach does not directly apply to more than two subsystems due to causality issues when designing the predictors. Interestingly, this limitation does not appear when considering only two subsystems. Moreover, compared to~\cite{auriol2022robust}, a state-observer is designed in this paper. 

\modif{
In this paper, we overcome the limitations of~\cite{redaud2021SCL} (regarding the dimension of the subsystems) and of~\cite{auriol2022robust} (regarding the number of subsystems) to design a stabilizing output-feedback controller for an interconnection of an arbitrary number of non-scalar PDE subsystems with a chain structure.   Although the methodology we propose uses the same ingredients presented in~\cite{auriol2022robust} and~\cite{redaud2021SCL}, extending such results to a chain of non-scalar subsystems is far from trivial. Indeed, when dealing with non-scalar systems, backstepping transformations cannot usually remove all the in-domain coupling terms~\cite{coron2017finite}. Therefore, due to these remaining in-domain coupling terms, the approach presented in~\cite{redaud2021SCL} does not work as the predictors cannot be adequately defined (causality problem). In the case of two subsystems, this problem was partially solved in~\cite{auriol2022robust} using appropriate flatness-based feedforward tracking components in the control input. However, this solution required adding additional terms to the backstepping transformation depending on the downstream subsystems, leading to complex, intricate kernel equations when there are more than two subsystems. We introduce a \textbf{new type of integral transformation}, with a time-affine component to overcome these limitations. This new component is used to ‘‘clear the actuation path" of each subsystem by removing the local terms initially present in the system and avoiding additional terms coming from the downstream subsystem. To the best of the author's knowledge, such a paper represents a novelty in the literature as it solves a challenging stabilization problem by taking advantage of a new class of time-affine transformations. Finally, it is the first time an observer is designed for a chain of non-scalar subsystems (only state-feedback stabilization was considered in~\cite{auriol2022robust}).
}

The paper is organized as follows. We introduce the class of system under consideration in Section~\ref{Section_Pres_Systems}. In Section~\ref{Chap_3_Sec_SF}, we present a new class of integral transformations and recursively design a stabilizing state-feedback control law. An observer that estimates delayed values of the states is proposed in Section~\ref{Chap_3_Sec_State_Estimation}. It is used to obtain an output-feedback stabilizing control law. The results of the paper are illustrated with simulations in Section~\ref{Section_Simulations}.  Finally, we give some concluding remarks in Section~\ref{Section_Conclusion}.
%that simplifies the structure of the system. %It then becomes possible to design a stabilizing state-feedback control law recursively. An observer that estimates delayed values of the states is proposed in Section~\ref{Chap_3_Sec_State_Estimation}. %This observer can be combined with the previous state-feedback controller to obtain an output-feedback stabilizing control law. The results of the paper are illustrated with simulations in Section~\ref{Section_Simulations}. Finally, we give some concluding remarks in Section~\ref{Section_Conclusion}.

\subsection{Notations}
In this section, we detail the notations used throughout this paper. For any distinct real numbers $a$ and $b$, any positive integer $n$, we denote $L^2([a,b],\mathbb{R}^n)$ the space of real-valued square-integrable functions defined on~$[a,b]$ with the standard~$L^2$ norm, \emph{i.e.}, for any $f\in L^2([a,b],\mathbb{R}^n)$,~$||f||_{L^2([a,b])}^2 = \int_a^b f^T(x)f(x) \mathrm{d}x$.
For $n \in \mathbb{N}$ functions $u_k$ in~$L^2([a,b],\mathbb{R})$, we define the $L^2$ norm of the vector $(u_1,\hdots,u_n)$ as the sum of the square of the $L^2$-norm of each function composing the vector: $||(u_1,\hdots,u_n)||_{L^2}^2=\sum_{k=1}^n||u_k||^2_{L^2([a,b])}.$
 %This notation can be straightforwardly generalized to an arbitrary finite number of $L^2$-functions.
The set~$L^{\infty}([0,1],\mathbb{R})$ denotes the space of bounded real-valued functions defined on~$[0,1]$ with the standard~$L^\infty$ norm, \textit{i.e.}, for any~$f \in L^\infty ([0,1],\mathbb{R})$, $||f||_{L^\infty}=\underset{x \in [0,1]}{\text{ess sup}}   |f(x)|.$   For any positive integer $n$, we denote $H^1([a,b],\mathbb{R}^n)$ the one-dimensional Sobolev space.
%, i.e. the subset of functions $f$ in $L^2([a, b], \mathbb{R}^n)$ such that $f$ and its weak derivative of order 1 have a finite $L^2$-norm. The associated norm is denoted $||\cdot||_{H^1}$. 
% For any $ f \in H^1([a,b],\mathbb{R}^n)$, $$||f||_{H^1}=\sqrt{\int_a^b f^T(x)f(x)+(f'(x))^T(f'(x)) dx}.$$
For any integer $m>0$ and any real delay $\tau>0$, we denote $L^2([-\tau,0], \mathbb{R}^m)$ the Banach space of $L^2$ functions mapping the interval $[-\tau, 0]$ into $\mathbb{R}^m$. For a function $\phi: [-\tau,\infty) \mapsto \mathbb{R}^m$, we define its partial trajectory $\phi_{[t]}$ by $\phi_{[t]}:\phi(t+\theta)$, $-\tau \leq \theta \leq 0$. This maximum delay $\tau$ will be related to the transport velocities of the considered PDE system. The associated norm is denoted $||\phi_{[t]}||_{L^2_\tau}$ where
% \begin{align}
% ||\phi_{[t]}||_{L^2_\tau}=\left(\int_{-\tau}^0 \phi^T(t+s)\phi(t+s)d s\right)^{\frac{1}{2}}. \label{norm_delay}
% \end{align}
for every $\tau>r$, we define
\begin{align}
||\phi_{[t]}||_{L^2_r}= \left(|\int_{-r}^0 \phi^T(t+\theta)\phi(t+\theta)d\theta|\right)^{\frac{1}{2}}.\label{norm_r}
\end{align}
The identity matrix of dimension $n$ will be denoted $\text{Id}_n$ (or $\text{Id}$ if no confusion on the dimensions arises).

\section{Problem under consideration} \label{Section_Pres_Systems}

\subsection{Interconnection with a cascade structure}
In this paper, we consider a system composed of $N>0$ PDE subsystems interconnected through their boundaries in a chain structure, as schematically represented in Figure~\ref{Fig_Example_PDE_chain}. The control input and the available sensors are located at one extremity of the chain. Each subsystem is composed of an arbitrary number of linear hyperbolic PDEs and is modeled by the following set of equations~($i \in \{1, \dots, N\}$)
\begin{align}
\partial_t u_i(t,x)+\Lambda_i^+ \partial_xu_i(t,x)&=\Sigma^{++}_i(x) v_i(t,x) \nonumber \\
&+\Sigma^{+-}_i(x) v_i(t,x), \label{Chap_3_eq_init_PDE_u}\\
\partial_tv_i(t,x)-\Lambda_i^-  \partial_xv_i(t,x)&=\Sigma^{-+}_i(x) u_i(t,x)\nonumber \\
&+\Sigma^{--}_i(x) v_i(t,x), \label{Chap_3_eq_init_PDE_v}
\end{align}
evolving in~$\{(t,x)~\text{s.t.}~t>0,~x\in[0,1] \}$, where $u_i=(u^1_i, \hdots, u^{n_i}_i)^T$ and $v_i=(v^1_i, \dots, v^{m_i}_i)^T$, all the $n_i$ and $m_i$ being positive integers. The matrices~$\Lambda_i^+$ and~$\Lambda_i^-$ are diagonal and represent the \textbf{transport velocities} of each subsystem. We have~$\Lambda_i^+=$ diag~$(\lambda_i^j)$ and~$\Lambda_i^-=$ diag ($\mu_i^j$) and we assume that their coefficients satisfy$$
-\mu_i^{m_i}< \cdots <-\mu_i^1 <0<\lambda_i^1 <\cdots < \lambda_i^{n_i}.$$
These transport velocities are assumed to be constant. However, all our results can be extended to space-dependent transport velocities at the cost of technical and lengthy computations. The spatially-varying coupling matrices~$\Sigma_i^{\cdot \cdot}$ are regular matrices (we assume here that each coefficient of the matrix is a continuous function). Without any loss of generality, we can assume that the diagonal entries of~$\Sigma_i^{++}$  and~$\Sigma_i^{--}$ are equal to zero~\cite{coron2013local}.
%. Indeed, these diagonal terms can be transferred to the anti-diagonal terms using an exponential change of variables~\cite{coron2013local}.
The different subsystems are connected through their boundaries in a chain structure. We have
\begin{align}
u_i(t,0) &= Q_{i,i}v_i(t,0)+Q_{i,i-1}u_{i-1}(t,1),
\label{Chap_3_eq_init_bound_u}\\
v_i(t,1) &=R_{i,i} u_i (t,1) +R_{i,i+1} v_{i+1}(t,0)\label{Chap_3_eq_init_bound_v}
\end{align}
where the different coupling and~$R_{i,j}, Q_{i,j}$ are constant. By convention we consider that~$R_{N,N+1}=0$ and~$Q_{1,0}=\text{Id}$. The function $u_{0}(t,1)$ corresponds to the \textbf{control input}, $U(t) \in \mathbb{R}^{n_1}$. The measured output is denoted as $y(t)$ and verifies $y(t)=u_N(t,1)$.  
The initial conditions of each subsystem belong to $H^1([0,1],\mathbb{R}^{n_i})\times H^1([0,1],\mathbb{R}^{m_i})$. They satisfy the appropriate compatibility conditions (as stated in~\cite{bastin2016stability}), so that the system~\eqref{Chap_3_eq_init_PDE_u}-\eqref{Chap_3_eq_init_bound_v} is well-posed~\cite[Theorem A.1]{bastin2016stability}. Finally, we denote $\tau_i$ the maximum transport delay associated to each PDE subsystem: $\tau_i=\frac{1}{\lambda_i^1}+\frac{1}{\mu_i^1}$.

The interconnected system~\eqref{Chap_3_eq_init_PDE_u}-\eqref{Chap_3_eq_init_bound_v} can be recast under a more condensed form as a general $n+m$ system, using a technique referred to as \textbf{folding} (see~\cite{auriol2020output,auriol2022robust,de2018backstepping} for details). However, such a system would still be underactuated and classical results from the literature could not be applied. Moreover, such a condensed representation would shadow the cascade structure between the different subsystems. Conversely, the representation~\eqref{Chap_3_eq_init_PDE_u}-\eqref{Chap_3_eq_init_bound_v}  highlights that the interactions between the different subsystems only occur at the boundaries. %For instance, the output of the second subsystem entering the first subsystem could be seen as a disturbance acting on the first subsystem (even if the corresponding output of the first subsystem indirectly modifies this disturbance signal). 
For a subsystem $i$, we will call the subsystem $i+1$ the \textbf{downstream subsystem} and the subsystem $i-1$ the \textbf{upstream subsystem}.

\begin{figure*}[htb]%
\begin{center}
 \scalebox{0.8}{
\begin{tikzpicture}
% (u_n,v_n) system
%PDE
\draw [>=stealth,->,red,very thick] (0,0) -- (3,0);
\draw [red] (1.5,0) node[above]{$u_1(t,x)$};
\draw [>=stealth,<-,blue,very thick] (0,-1.5) -- (3,-1.5);
\draw [blue] (1.5,-1.5) node[below]{$v_1(t,x)$};

%In-domain couplings 
\draw [>=stealth,<-,dashed, thick] (1,-1.5) -- (1,0);
\draw(1,-0.75) node[left]{$\Sigma^{-+}_1$};
\draw [>=stealth,->,dashed, thick] (2,-1.5) -- (2,0);
\draw(2,-0.75) node[left]{$\Sigma^{++}_1$};

%B.C
\draw [blue,>=stealth, thick](-0.5,-0.75) arc (-180:-135:1.1);
\draw [blue,>=stealth,->, thick](-0.5,-0.75) arc (-180:-225:1.1);
\draw [blue] (-0.5,-0.75) node[left]{$Q_{1,1}$};

\draw [red,>=stealth, thick](3.5,-0.75) arc (0:45:1.1);
\draw [red,>=stealth,->, thick](3.5,-0.75) arc (0:-45:1.1);
\draw [red] (3.5,-0.75) node[left]{$R_{1,1}$};

%B.C between sub (n) and (n-1)
\draw [>=stealth,->,green!50!black!90,very thick] (3.5,0) -- (4.5,0);
\draw [color=green!50!black!90] (4,0) node[above]{$Q_{2,1}$};

\draw [>=stealth,<-,green!50!black!90,very thick] (3.5,-1.5) -- (4.5,-1.5);
\draw [color=green!50!black!90] (4,-1.5) node[below]{$R_{1,2}$};

%% Between the subsystems
\draw (5,0) node {$\bullet$};
\draw (5.5,0) node {$\bullet$};
\draw (6,0) node {$\bullet$};

\draw [blue,>=stealth, thick](4.5,-0.75) arc (-180:-135:1.1);
\draw [blue,>=stealth,->, thick](4.5,-0.75) arc (-180:-225:1.1);
%\draw [red] (4.5,-0.75) node[right]{$\rho_{n-1,n-1}$};

\draw [red,>=stealth, thick](6.5,-0.75) arc (0:45:1.1);
\draw [red,>=stealth,->, thick](6.5,-0.75) arc (0:-45:1.1);
%\draw [red] (6.5,-0.75) node[left]%{$\rho_{3,3}$};

\draw (5,-1.5) node {$\bullet$};
\draw (5.5,-1.5) node {$\bullet$};
\draw (6,-1.5) node {$\bullet$};

\draw [>=stealth,->,green!50!black!90,very thick] (6.5,0) -- (7.5,0);
\draw [color=green!50!black!90] (7,0) node[above]{$Q_{n-1,n-2}$};

\draw [>=stealth,<-,green!50!black!90,very thick] (6.5,-1.5) -- (7.5,-1.5);
\draw [color=green!50!black!90] (7,-1.5) node[below]{$R_{n-2,n-1}$};
%% Subsystem 2
% (u_2,v_2) system
%PDE
\draw [>=stealth,->,red,very thick] (8,0) -- (11,0);
\draw [red] (9.5,0) node[above]{$u_{n-1}(t,x)$};
\draw [>=stealth,<-,blue,very thick] (8,-1.5) -- (11,-1.5);
\draw [blue] (9.5,-1.5) node[below]{$v_{n-1}(t,x)$};

%In-domain couplings 
\draw [>=stealth,<-,dashed, thick] (9,-1.5) -- (9,0);
\draw(9,-0.75) node[left]{$\Sigma^{-+}_{n-1}$};
\draw [>=stealth,->,dashed, thick] (10,-1.5) -- (10,0);
\draw(10,-0.75) node[left]{$\Sigma^{+-}_{n-1}$};

%B.C
\draw [blue,>=stealth, thick](7.5,-0.75) arc (-180:-135:1.1);
\draw [blue,>=stealth,->, thick](7.5,-0.75) arc (-180:-225:1.1);
%\draw [blue] (7.5,-0.75) node[right]{$q_{n-1,n-1}$};

\draw [red,>=stealth, thick](11.5,-0.75) arc (0:45:1.1);
\draw [red,>=stealth,->, thick](11.5,-0.75) arc (0:-45:1.1);
%\draw [red] (11.5,-0.75) node[left]{$\rho_{n-1,n-1}$};

%B.C between sub (n) and (n-1)

\draw [>=stealth,->,green!50!black!90,very thick] (11.5,0) -- (12.5,0);
\draw [color=green!50!black!90] (12,0) node[above]{$Q_{n,n-1}$};

\draw [>=stealth,<-,green!50!black!90,very thick] (11.5,-1.5) -- (12.5,-1.5);
\draw [color=green!50!black!90] (12,-1.5) node[below]{$R_{n-1,n}$};

%%Subsystem 1

%PDe
\draw [>=stealth,->,red,very thick] (13,0) -- (16,0);
\draw [red] (14.5,0) node[above]{$u_n(t,x)$};
\draw [>=stealth,<-,blue,very thick] (13,-1.5) -- (16,-1.5);
\draw [blue] (14.5,-1.5) node[below]{$v_n(t,x)$};

%In-domain couplings 
\draw [>=stealth,<-,dashed, thick] (14,-1.5) -- (14,0);
\draw(14,-0.75) node[left]{$\Sigma^{-+}_n$};
\draw [>=stealth,->,dashed, thick] (15,-1.5) -- (15,0);
\draw(15,-0.75) node[left]{$\Sigma^{+-}_n$};

%B.C
\draw [blue,>=stealth, thick](12.5,-0.75) arc (-180:-135:1.1);
\draw [blue,>=stealth,->, thick](12.5,-0.75) arc (-180:-225:1.1);
\draw [blue] (12.5,-0.75) node[left]{$Q_{n,n}$};

\draw [red,>=stealth, thick](16.5,-0.75) arc (0:45:1.1);
\draw [red,>=stealth,->, thick](16.5,-0.75) arc (0:-45:1.1);
\draw [red] (16.5,-0.75) node[left]{$R_{n,n}$};

\draw [>=stealth,->,purple,very thick] (-0.8,0) -- (-0.4,0);
\draw [purple] (-0.8,0) node[left]{$U(t)$};

%% x-axis

\draw [>=stealth,<->,very thick] (7.5,-2.5) -- (11.5,-2.5);
\draw [>=stealth,<->,very thick] (12.5,-2.5) -- (16.5,-2.5);
\draw [very thick] (0,-2.6) -- (0,-2.4);
\draw [very thick] (11,-2.6) -- (11,-2.4);
\draw [very thick] (13,-2.6) -- (13,-2.4);
\draw [very thick] (8,-2.6) -- (8,-2.4);
\draw [very thick] (16,-2.6) -- (16,-2.4);
\draw (11,-2.4) node[above]{1};
\draw (13,-2.4) node[above]{0};
\draw (8,-2.4) node[above]{0};
\draw [very thick] (3,-2.6) -- (3,-2.4);
\draw (3,-2.4) node[above]{1};
\draw (16,-2.4) node[above]{1};
\draw [>=stealth,<->,very thick] (-0.5,-2.5) -- (3.5,-2.5);
\draw (0,-2.4) node[above]{0};
\draw [>=stealth,<->,very thick] (4,-2.5) -- (7,-2.5);
%\draw (17.5,-2.4) node[above]{$x$};
\end{tikzpicture}}
\end{center}
\caption{Schematic representation of the chain of linear PDE subsystem~\eqref{Chap_3_eq_init_PDE_u}-\eqref{Chap_3_eq_init_PDE_v}.}%
\label{Fig_Example_PDE_chain}%
\end{figure*}
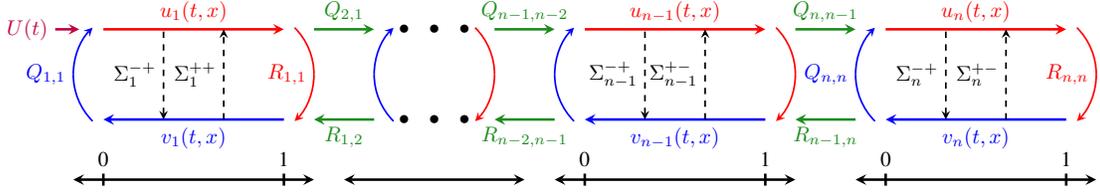

%%%%%%%%%%%%%%%%%%%%%%%%%%%%%%%%%%%%%%%%%%%%%%%%%%%%%%%
\subsection{Structural assumptions}
To design an appropriate stabilizing output feedback controller, we require several assumptions. \modif{First, to guarantee the possibility of designing a delay-robust controller, we must avoid having an infinite number of unstable roots in the right-half plan (as shown in~\cite{logemann1996conditions})}
%, as this prevents the possibility of (delay)-robust stabilization (as shown~\cite{logemann1996conditions}). 
\modif{This induces the following assumption (see~\cite{auriol2019explicit})
\begin{assumption} \label{Ass_robustness}
    The open-loop system~\eqref{Chap_3_eq_init_PDE_u}-\eqref{Chap_3_eq_init_bound_v} without the in-domain coupling terms $\Sigma_i^{\cdot \cdot}$ is exponentially stable.
\end{assumption}
} Then, we need the following assumption to stabilize the downstream subsystem states using actuation from the upstream subsystem.
\begin{assumption} \label{Chap_3_ass-control}
For all $i \in \{2,\dots, N\}$, the rank of the matrix $Q_{i-1,i}$ is equal to $n_i$.
\end{assumption}
This assumption implies that the matrix $Q_{i-1,i}$ admits a right inverse. A possible choice is given by the Moore–Penrose right
inverse: $Q_{i-1,i}^\top(Q_{i-1,i}Q_{i-1,i}^\top)^{-1}$.  This \textbf{conservative} assumption will be used to design a \emph{virtual actuation} for each subsystem. 
\modif{It implies that the dimension of the (virtual) input entering the subsystem $i$ (which corresponds to the effect of the upstream subsystem) is equal to the number of the rightward propagating states $n_i$ of this subsystem. Such a condition is usually required to design stabilizing controllers for hyperbolic systems that do not have a specific structure (see, e.g., \cite{coron2017finite,auriol2016minimum,hu2016control}). Indeed, to the best of our knowledge, only marginal results currently exist in the literature for stabilizing under-actuated systems (i.e., systems for which the dimension of the control input is smaller than the dimension of the boundary
 state) with no specific cascade structure (see~\cite{auriol2020CDCunderactuation}). All in all, Assumption~\ref{Chap_3_ass-control} is required to avoid such an underactuated configuration.  %Moreover, if it were possible to design a stabilizing controller for under-actuated systems, then we could directly apply such results to design a stabilizing controller for the interconnected system~\eqref{Chap_3_eq_init_PDE_u}-\eqref{Chap_3_eq_init_bound_v}. 
%However, such results do not currently exist, which explains why 
Therefore, to stabilize the system~\eqref{Chap_3_eq_init_PDE_u}-\eqref{Chap_3_eq_init_bound_v}, we will design a specific control strategy that takes advantage of the interconnection structure. }
 We are led to a similar assumption to designing a state observer.
\begin{assumption}\label{Chap_3_ass-obs}
For all $i \in \{2,\dots, N\}$, the rank of the matrix $Q_{i-1,i}$ is equal to $n_{i-1}$.
\end{assumption}
This assumption implies that the matrix $Q_{i-1,i}$ admits a left inverse. \modif{This condition was not present in~\cite{auriol2022robust} (as only state-feedback stabilization was considered) since it is required to recursively design the proposed observer.} Again, Assumption~\ref{Chap_3_ass-obs} is conservative \modif{and is used to avoid having under-measured subsystems.} %but to the best of our knowledge, only specific results currently exist in the literature for estimating under-measured systems with no specific cascade structure (see~\cite{auriol2020CDCunderactuation}). 
Combining Assumption~\ref{Chap_3_ass-control} and Assumption~\ref{Chap_3_ass-obs}, we obtain that all the $n_i$ are equal \modif{(i.e., all the subsystems have the same number of rightward propagating states)} and that the matrices $Q_{i-1,i}$ are invertible. This is related to the fact that we considered anti-collocated measurements. In the case of collocated measurements (i.e., $y(t)=v_1(t,0)$), Assumption~\ref{Chap_3_ass-obs} would have been expressed as a rank condition on the matrices $R_{i,i+1}$ (as it is the case in~\cite{redaud2021SCL} \modif{and the rightward propagating states would not need to have the same dimensions anymore}).

\subsection{Toward a recursive design} \label{Chap_3_rec_design}
The objective of this paper is to design an output-feedback control law that that stabilizes the interconnected system~\eqref{Chap_3_eq_init_PDE_u}-\eqref{Chap_3_eq_init_bound_v} in the sense of the $L^2$-norm. From Figure~\ref{Fig_Example_PDE_chain}, it can be seen that a subsystem~$i$ will act on the downstream subsystem~$i+1$ through~$u_i(t,1)$, and on the upstream subsystem~$i-1$ through~$v_i(t,0)$. Thus, each subsystem can only be stabilized through its upstream subsystem and estimated through its downstream subsystem. Due to the hyperbolic nature of the different subsystems, the effect of the control input $U(t)$ on the subsystem $i$ will be delayed and modified by the different in-domain coupling terms. 

To stabilize the whole chain, we extend the \textit{recursive interconnected dynamics control framework} introduced in~\cite{redaud2021SCL}. Roughly speaking, the control law is recursively obtained by considering stabilizing virtual inputs for each subsystem and ensuring the output of the upstream subsystem converges to this desired virtual input. The control design becomes more straightforward and is based on simple assumptions that can be independently verified for each subsystem. \modif{Therefore, the proposed recursive design is somehow inspired by the classical integrator backstepping. Such connections were, for instance, already mentioned in~\cite{gehringsystematic} for an ODE-PDE-ODE interconnection when using an analogous approach.}
We propose the following control strategy: 
\begin{enumerate}
\item First, using integral transformations, we simplify the structure of each subsystem to remove the in-domain coupling terms that appear in the $u_i$-PDEs (equation~\eqref{Chap_3_eq_init_PDE_u}). \modif{The transformations proposed in this paper are new as they include a time-affine component.}
\item Then, for each subsystem $i$, we consider the effect of the upstream subsystem $i-1$ as a delayed virtual input $U_i(t-\sum_{j=1}^{i-1}\frac{1}{\lambda_j^1})$ and the effect of the downstream subsystem $i+1$ as a disturbance term. We combine appropriate \textbf{state predictions} with a \textbf{flatness-based feedforward tracking controller} to guarantee that the right output of this subsystem converges to the delayed virtual input $U_{i+1}(t-\sum_{j=1}^{i}\frac{1}{\lambda_j^1})$ that will stabilize the downstream subsystem. Iterating such a procedure, it is possible to design a stabilizing control law $U(t)$ for the whole system. 
\item The closed-loop stability is shown recursively, using \textbf{Input-to-State Stability} (ISS) properties. 
\item A similar \textbf{recursive approach} is used to design a state observer. Going recursively from one subsystem to the next, we can estimate delayed values of the states.% at each subsystem boundary. 
\item Finally, similarly to what has been done for finite-dimensional systems~\cite{karafyllis2017predictor}, we can adjust the state predictors %used in the control design 
to obtain an output-feedback controller. 
\end{enumerate}
The proposed framework allows for a ``plug-and-play"-like approach to control design since additional subsystems satisfying similar conditions can be added to the network using the same procedure. 
% Moreover, it offers interesting perspectives as it can be applied to different classes of interconnected systems. A promising extension has, for instance, been suggested in~\cite{xu2023stabilization} for parabolic systems. 

%%%%%%%%%%%%%%%%%%%%%%%%%%%%%%%%%%%%%%%%%%%%%%
\section{State-feedback controller} \label{Chap_3_Sec_SF}
%%%%%%%%%%%%%%%%%%%%%%%%%%%%%%%%%%%%%
\subsection{Backstepping transformations}
The first objective before applying our recursive control strategy is to simplify the structure of the interconnected system~\eqref{Chap_3_eq_init_PDE_u}-\eqref{Chap_3_eq_init_bound_v} in order to ‘‘clear the actuation path" of each subsystem by removing the local terms initially present in equation~\eqref{Chap_3_eq_init_PDE_u}. \modif{In the case of two subsystems, this was done in~\cite{auriol2022robust} using a specific backstepping transformation adjusted from~\cite{hu2015boundary}, since due to the interconnection between the $i^{th}$ subsystem and the downstream subsystem $(i+1)$, the backstepping transformation given in~\cite{hu2015boundary} displays additional terms (depending on the state $v_{i+1}(t,0)$) that can cause causality issues when designing the control law. Unfortunately, for more than two subsystems, they cannot be straightforwardly removed by adjusting the transformation given in~\cite{auriol2022robust} without adding stringent conditions on the boundary coupling terms (more restrictive than Assumption~\ref{Chap_3_ass-control} and Assumption~\ref{Chap_3_ass-obs}).}
%overcomes this limitation (in the case of two subsystems) by adding to the backstepping transformation from~\cite{hu2015boundary} an additional component that depends on the downstream state $u_{i+1}$ and $v_{i+1}$. However, extending such a transformation to more than two subsystems requires (for each subsystem) the addition of components in the backstepping transformation that depend on all the downstream states. This leads to complex, intricate kernel equations for which we could only prove the well-posedness by adding stringent conditions on the boundary coupling terms (more restrictive than Assumption~\ref{Chap_3_ass-control} and Assumption~\ref{Chap_3_ass-obs}). These observations are similar to what was already noticed in~\cite{auriol2020output} when applying this class of integral transformations to a chain of scalar subsystems.}
\modif{To avoid such conservative conditions, we consider a new class of transformations. For each subsystem $i$, we combine a classical backstepping Volterra transformation (inspired from~\cite{auriol2016minimum}) with an integral term that depends on delayed values of the downstream state $v_{i+1}(t,0)$.}
More precisely, for all $i \in \{1,\hdots,N\}$, and all $t\geq \frac{1}{\lambda_i^1}$, we consider the integral transformation defined by 
\begin{align}
    &\alpha_i(t,x)=u_i(t,x)+\int_0^{\frac{x}{\lambda_i^1}} F_{i}(x,y) v_{i+1}(t-y,0)dy \nonumber \\&+\int_x^1 K_i^{uu}(x,y)u_i(t,y)+K_i^{uv}(x,y)v_i(t,y)dy, \label{Chap_3_eq_integral_time_terms} \\
    &\beta_i(t,x)=v_i(t,x)+\int_x^1 K_i^{vu}(x,y)u_i(t,y)dy\nonumber \\
    &\quad \quad +\int_x^1K_i^{vv}(x,y)v_i(t,y)dy,\label{Chap_3_eq_integral_time_terms_2}
\end{align}
where the kernels $K_i^{\cdot\cdot}$ are piecewise continuous functions defined on $\mathcal{T}_u=\{(x,y) \in [0,1]^2~|~x \leq y\}$, while the kernels $F_i$ are piecewise continuous functions defined on the triangular domain $\{(x,y) \in [0,1] \times [0,\frac{1}{\lambda_i^1}],~y \leq \frac{x}{\lambda_i^1}\}$. By convention $F_{N+1}=0$.  The kernels  $K_i^{\cdot\cdot}$ and $K_i^{\cdot\cdot}$ verify the following set of PDEs
\modif{
\begin{align}
    \Lambda_i \partial_x K_i +\partial_y K_i \Lambda_i &= - K_i\Sigma_i(y)+\left(\begin{smallmatrix}G_i(x) & 0 \\ 0 & 0\end{smallmatrix}\right)K_i, \label{Chap_3 eq_kernel_K_i_uu} 
\end{align}
where $\Lambda_i = \text{diag} (\Lambda_i^+, -\Lambda_i^-)$, $\Sigma_i = \left(\begin{smallmatrix} \Sigma_i^{++}&\Sigma_i^{+-}\\\Sigma_i^{-+}&\Sigma_i^{--}\end{smallmatrix}\right)$ and $K_i = \left(\begin{smallmatrix} K_i^{uu}&K_i^{uv}\\K_i^{vu}&K_i^{vv}\end{smallmatrix}\right)$
and with the boundary conditions
% \small
% \begin{align}
% &\Lambda^+_i \partial_x K_i^{uu}(x,y)+\partial_y K_i^{uu}(x,y)\Lambda_i^+=-K_i^{uu}(x,y)\Sigma^{++}_i(y)\nonumber \\&-K_i^{uv}(x,y)\Sigma^{-+}_i(y)+G_i(x)K_i^{uu}(x,y),  \\
% &\Lambda^+_i \partial_x K_i^{uv}(x,y)-\partial_y K_i^{uv}(x,y)\Lambda_i^-=-K_i^{uu}(x,y)\Sigma^{+-}_i(y)\nonumber \\
% &-K_i^{uv}(x,y)\Sigma^{--}_i(y)+G_i(x)K_i^{uv}(x,y), \\
% &\Lambda^-_i \partial_x K_i^{vu}(x,y)-\partial_y K_i^{vu}(x,y)\Lambda_i^+=K_i^{vu}(x,y)\Sigma^{++}_i(y)\nonumber \\
% &+K_i^{vv}(x,y)\Sigma^{-+}_i(y),  \\
% &\Lambda^-_i \partial_x K_i^{vv}(x,y)+\partial_y K_i^{vv}(x,y)\Lambda_i^-=K_i^{vu}(x,y)\Sigma^{+-}_i(y)\nonumber \\
% &+K_i^{vv}(x,y)\Sigma^{--}_i(y),\label{Chap_3 eq_kernel_K_i_vv} 
% \end{align}
% \normalsize
%with the boundary conditions
\begin{align}
    &\Lambda_i K_i(x,x) - K_i(x,x)\Lambda_i = \Sigma_i(x)-\left(\begin{smallmatrix}G_i(x) & 0 \\ 0 & 0\end{smallmatrix}\right),\label{Chap_3_ker_K_BC_1}\\
&K_i^{uu}(x,1)\Lambda_i^+=K_i^{uv}(x,1)\Lambda^-_iR_{ii},\label{Chap_3_ker_K_2}
\end{align}}
% \begin{align}
% &\Lambda_i^+K_i^{uu}(x,x)-K_i^{uu}(x,x)\Lambda_i^+=\Sigma_i^{++}(x)-G_i(x), \\
% &\Lambda_i^+K_i^{uv}(x,x)+K_i^{uv}(x,x)\Lambda_i^-=\Sigma_i^{+-}(x),\\
% &\Lambda_i^-K_i^{vu}(x,x)+K_i^{vu}(x,x)\Lambda_i^+=-\Sigma_i^{-+}(x),\\
% &\Lambda_i^-K_i^{vv}(x,x)-K_i^{vv}(x,x)\Lambda_i^-=-\Sigma_i^{--}(x),\\
% &K_i^{uu}(x,1)\Lambda_i^+=K_i^{uv}(x,1)\Lambda^-_iR_{ii},\label{Chap_3_ker_K_2}
% \end{align}
where $G_i(x)$ is a piecewise continuous \textbf{strictly upper-triangular} matrix function defined on $[0,1]$ through the first boundary condition of~\eqref{Chap_3_ker_K_BC_1}. More precisely, for all $1\leq k,\ell \leq n_i $, the boundary condition $\Lambda_i^+K_i^{uu}(x,x)-K_i^{uu}(x,x)\Lambda_i^+=\Sigma_i^{++}(x)-G_i(x)$ rewrites 
\begin{align}
&(G_i(x))_{k\ell}=(\Sigma_i^{++})_{k\ell}+ (\lambda_i^\ell- \lambda_i^k)(K_i^{uu}(x,x))_{k\ell}, \quad \text{if $k \leq \ell$}, \nonumber \\
&(K_i^{uu}(x,x))_{k\ell}=\frac{(\Sigma_i^{++})_{k\ell}}{\lambda_i^k-\lambda_i^\ell}, \quad \text{if $k > \ell$}. \label{Chap_3_ker_K_last}
\end{align}
It is important to emphasize that the matrix $G_i$ is strictly upper-triangular since for $k=\ell$, we have $(G_i(x))_{k\ell}=0$ (as the diagonal entries of $\Sigma_i^{++}$ are equal to zero). To these boundary conditions we add arbitrary boundary conditions for  $(K_i^{vv}(0,y))_{k\ell}$ when $k < \ell$, and arbitrary conditions for $(K_i^{vv}(x,1))_{k\ell}$ when $\ell \leq k$. \modif{With these additional boundary conditions,} the set of kernel equations~\eqref{Chap_3 eq_kernel_K_i_uu}-\eqref{Chap_3_ker_K_last} admits a unique piecewise continuous solution~\cite{auriol2016minimum}. The kernels $F_i$ verify
\begin{align}
&\Lambda^+_i \partial_x F_i(x,y)+\partial_y F_i(x,y)=G_i(x)F_i(x,y), \label{Chap_3_F_PDE}\\
&F_i(x,0)=-K_i^{uv}(x,0)\Lambda^-_iR_{i,i+1}, \\
&(F_i(x,\frac{x}{\lambda_i^1}))_{k\ell}=0, \quad 1< k \leq n_{i},~1\leq \ell \leq m_{i+1}. \label{Chap_3_F_BC}
\end{align}
\modif{Applying~\cite[Theorem 3.2]{di2018stabilization} (on the triangular domain $\{(x,y) \in [0,1] \times [0,\frac{1}{\lambda_i^1}],~y \leq \frac{x}{\lambda_i^1}\}$), one can show that equations~\eqref{Chap_3_F_PDE}-\eqref{Chap_3_F_BC} admit a unique piecewise continuous solution.} The transformation~\eqref{Chap_3_eq_integral_time_terms}-\eqref{Chap_3_eq_integral_time_terms_2} is a Volterra transformation to which an affine term that depends on the state $v_{i+1}$ is added. Consequently, it is invertible \cite{yoshida1960lectures} and there exist piecewise continuous functions $L_i^{\cdot\cdot}$ defined on $\mathcal{T}_u$ and piecewise continuous functions $H_i^{\cdot}$ defined on the rectangular domain  $\{(x,y) \in [0,1] \times [0,\frac{1}{\lambda_i^1}]\}$ such that for all $t\geq \frac{1}{\lambda_i^1}$,
 \begin{align}
    &u_i(t,x)=\alpha_i(t,x)+\int_0^{\frac{1}{\lambda_i^1}} H^\alpha_{i}(x,y) v_{i+1}(t-y,0)dy\nonumber \\&+\int_x^1 L_i^{\alpha\alpha}(x,y)\alpha_i(t,y)+L_i^{\alpha \beta}(x,y)\beta_i(t,y)dy, \label{Chap_3_eq_integral_time_terms_inverse} \\
    &v_i(t,x)=\beta_i(t,x)+\int_0^{\frac{1}{\lambda_i^1}} H^\beta_{i}(x,y) v_{i+1}(t-y,0)dy\nonumber \\&+\int_x^1 L_i^{\beta\alpha}(x,y)\alpha_i(t,y)+L_i^{\beta \beta}(x,y)\beta_i(t,y)dy. \label{Chap_3_eq_integral_time_terms_inverse_2}
\end{align}
\modif{Compared to~\eqref{Chap_3_eq_integral_time_terms}-\eqref{Chap_3_eq_integral_time_terms_2}, the inverse transformation~\eqref{Chap_3_eq_integral_time_terms_inverse}-\eqref{Chap_3_eq_integral_time_terms_inverse_2} has integral components involving $v_{i+1}(t-y,0)$ on both equations. Moreover, the upper limits of these integrals are $\frac{1}{\lambda_i^1}$ instead of $\frac{x}{\lambda_i^1}$. This can be seen by applying the inverse Volterra transformation to the vector~$\left(\begin{smallmatrix}\int_0^{\frac{x}{\lambda_i^1}} F_{i}(x,y) v_{i+1}(t-y,0)dy \\ 0
\end{smallmatrix}\right)$, and applying Fubini's theorem. However, we emphasize that the kernels $H_i^\alpha$ and $H_i^\beta$ may be equal to zero on some parts of the rectangular domains $[0,1]\times[0,\frac{1}{\lambda_i^1}]$. }
\modif{\begin{remark}
    Interestingly, the transformation~\eqref{Chap_3_eq_integral_time_terms}-\eqref{Chap_3_eq_integral_time_terms_2} shares several features with the one introduced in~\cite{redaud2024domain}, as they both combine a classical backstepping Volterra transformation with a component that depends on delayed values of the (downstream)  state. However, their nature and purpose are fundamentally different: 
    \begin{itemize}
    \item In~\cite{redaud2024domain}, the objective is to map the initial PDE system to any target system with a similar structure but whose source terms can be arbitrarily chosen. The triangular time-affine component of the integral transformation is used to modify the remaining in-domain coupling terms in the target system; 
    \item Here, the component $\int_0^{\frac{x}{\lambda_i^1}}F_i(x,y)v_{i+1}(t-y,0)$ of the transformation~\eqref{Chap_3_eq_integral_time_terms} is used to avoid displaying in the actuation path of the $i^{th}$-subsystem additional terms depending on $v_{i+1}(t, 0)$. The kernels $F_i$ do not have a triangular structure, and the kernel equations are completely different from the ones given in~\cite{redaud2024domain}.
\end{itemize}
All in all, the transformation~\eqref{Chap_3_eq_integral_time_terms}-\eqref{Chap_3_eq_integral_time_terms_2} and the one introduced in~\cite{redaud2024domain} can be seen as analogous tools applied in different contexts to overcome some limitations of the classical backstepping Volterra transformations.
\end{remark}}
\color{black}
\subsection{Target system}
\color{black}
For all $t\geq \frac{1}{\lambda_i^1}$, the transformation~\eqref{Chap_3_eq_integral_time_terms}-\eqref{Chap_3_eq_integral_time_terms_2} maps the system~\eqref{Chap_3_eq_init_PDE_u}-\eqref{Chap_3_eq_init_bound_v} to the target system
\begin{align}
    \partial_t \alpha_i(t,x)+\Lambda_i^+ \partial_x \alpha_i(t,x)&=G_i(x) \alpha_i(t,x), \label{Chap_3_PDE_alpha}\\
\partial_t\beta_i(t,x)-\Lambda_i^-  \partial_x\beta_i(t,x)&=\bar G_i(x)\alpha_i(t,1)\nonumber \\
&+\bar f_i(x)v_{i+1}(t,0), \label{Chap_3_PDE_beta}
\end{align}
\modif{
with the boundary conditions 
\begin{align}
&\alpha_i(t,0) = Q_{i,i}v_i(t,0)+Q_{i,i-1}\alpha_{i-1}(t,1)\nonumber \\
&+\int_0^1 K_i^{uu}(0,y)u_i(t,y)+K_i^{uv}(0,y)v_i(t,y)dy \nonumber \\
&-Q_{i,i-1}\int_0^{\frac{1}{\lambda_{i-1}^1}}F_{i-1}(1,y)v_i(t-y,0)dy,
\label{Chap_3_bound_alpha}\\
&\beta_i(t,1) =R_{i,i} \alpha_i (t,1) +R_{i,i+1} v_{i+1}(t,0)\nonumber \\
&-R_{i,i}\int_0^{\frac{1}{\lambda_i^1}}F_i(1,y)v_{i+1}(t-y,0)dy,\label{Chap_3_bound_beta}
\end{align}
where} $\bar G_i(x)=K_i^{vv}(x,1)\Lambda_i^-R_{i,i}-K_i^{vu}(x,1)\Lambda_i^+$ and $\bar f_i(x)=K_i^{vv}(x,1)\Lambda_i^-R_{i,i+1}$. By convention, we have $F_0=0$ and $Q_{1,0}\alpha_{0}(t,1)=U(t)$. The in-domain coupling terms appearing in equation~\eqref{Chap_3_eq_init_PDE_u} now have a triangular structure. In equation~\eqref{Chap_3_eq_init_bound_v}, all the local terms have been replaced by non-local terms that depend on $\alpha_i(t,1)$ and $v_{i+1}(t-\frac{x}{\lambda_i^1},0)$. 
% This structure will simplify the design of the proposed recursive stabilizing control law. For instance, in the case of a single subsystem ($N=1$), it becomes straightforward to design a stabilizing boundary feedback control law by canceling all the terms at the actuated boundary~\eqref{Chap_3_bound_alpha} (as shown in~\cite{auriol2016minimum}). 
\modif{Using the inverse transformation~\eqref{Chap_3_eq_integral_time_terms_inverse}-\eqref{Chap_3_eq_integral_time_terms_inverse_2}, it may be possible to substitute the remaining  $u_i$ and $v_i$ terms that appear in the target system~\eqref{Chap_3_PDE_alpha}-\eqref{Chap_3_bound_beta} by $\alpha_i$- and $\beta_i$-terms. However, this transformation also induces the appearance of $v_{i+1}(t,0)$-terms. Although it is possible to apply the transformation~\eqref{Chap_3_PDE_alpha} recursively-\eqref{Chap_3_bound_beta} (till we reach the last subsystem $N$) to get rid of all these $v_j(t,0)$ terms, the resulting expression would be cumbersome and involve intricate sums depending on delayed version of the state $\alpha_j$ and $\beta_j$. Therefore, we decided not to express these terms as functions of $\alpha_i$ and $\beta_i$ to increase readability. This will not affect the proposed analysis.}
In the next sections, we state several elementary properties for the system~\eqref{Chap_3_PDE_alpha}-\eqref{Chap_3_bound_beta}. We will then combine these properties to design our recursive stabilizing controller. For $t>\max_i\frac{1}{\lambda_i^1}$ and all $1\leq i \leq N$, all the integral transformations~\eqref{Chap_3_eq_integral_time_terms}-\eqref{Chap_3_eq_integral_time_terms_2} are well defined.

%%%%%%%%%%%%%%%%%%%%%%%%%%%%%%%%%%%%%
\subsection{Output trajectory tracking} \label{Sec_tracking_ISS}
Consider the $i^{th}$ subsystem composing the interconnection~\eqref{Chap_3_PDE_alpha}-\eqref{Chap_3_bound_beta}. Let us define the virtual control input acting on this subsystem as  
\begin{align}
\hat U_i(t)=Q_{i,i-1}\alpha_{i-1}(t+\sum_{j=1}^{i-1}\frac{1}{\lambda^1_j},1). \label{Chap_3_virtual_input} 
\end{align}
This virtual control input represents the action of the upstream subsystem on the subsystem $i$. %We do not choose $\hat U_i(t)=Q_{i,i-1}\alpha_{i-1}(t,1)$ as the virtual input, to guarantee the causality of the final controller. %This will be made clear in Section~\ref{Sec_recursive_State_feedback}. 
%Indeed, 
The delay $\sum_{j=1}^{i-1}\frac{1}{\lambda^1_j}$ corresponds to the total largest transport time between the control input $U(t)$ and the subsystem~$i$. It reflects the fact that the control input cannot directly act on the subsystem $i$, but that its effect is subject to this delay. %$\sum_{j=1}^{i-1}\frac{1}{\lambda^1_j}$. 
Equation~\eqref{Chap_3_bound_alpha} rewrites 
\begin{align}
\alpha_i(t,0) &= Q_{i,i}v_i(t,0)+\hat U_i(t-\sum_{j=1}^{i-1}\frac{1}{\lambda^1_j})\nonumber \\
&+\int_0^1 K_i^{uu}(0,y)u_i(t,y)+K_i^{uv}(0,y)v_i(t,y)dy \nonumber \\
& -Q_{i,i-1}\int_0^{\frac{1}{\lambda_{i-1}^1}}F_{i-1}(1,y)v_i(t-y,0)dy. \label{Chap_3_eq_alpha_i_0_rewritten}
\end{align}
We have the following property that guarantees the possibility for each subsystem to track any arbitrary function as long as predictions of the different states are available 
\begin{property}\label{Chap_3_assum_tracking}
Consider the $i^{th}$ subsystem~\eqref{Chap_3_PDE_alpha}-\eqref{Chap_3_bound_beta} $(i\in \{1,\hdots,N-1\})$ and define $\zeta_i$ an arbitrary known $H^1([0,\infty),\mathbb{R}^{n_{i+1}})$ function. Assume that there exists $t_0>0$ such that for all $t>t_0$ and all $x \in [0,1]$, it is possible to obtain a $\sum_{j=1}^{i-1}\frac{1}{\lambda_j^1}$ -ahead of time prediction of the PDE states $u_i(t,x)$, $v_i(t,x)$, $\alpha_i(t,x)$, $\beta_i(t,x)$, i.e. there exist predictor functions $P_{u_i}$, $P_{v_i}$, $P_{\alpha_i}$ and $P_{\beta_i}$  such that for all $t>t_0$ and all $x \in [0,1]$,
$P_{u_i}(t,x)=u_i(t+\sum_{j=1}^{i-1}\frac{1}{\lambda^1_j},x),$ $P_{v_i}(t,x)=v_i(t+\sum_{j=1}^{i-1}\frac{1}{\lambda^1_j},x),$ $P_{\alpha_i}(t,x)=\alpha_i(t+\sum_{j=1}^{i-1}\frac{1}{\lambda^1_j},x),$ $P_{\beta_i}(t,x)=\beta_i(t+\sum_{j=1}^{i-1}\frac{1}{\lambda^1_j},x)$.
 Then, there exists a control law $\hat U_i(t)$ such that for any $t>t_0+\frac{1}{\lambda_i^1}$, we have $\alpha_i(t,1)=\zeta_i(t)$. Moreover, if $\zeta_i(t) \equiv 0$, and $v_{i+1}(t) \equiv 0$, then, such a control law exponentially stabilizes the $i^{th}$ subsystem. 
   \end{property}
   \begin{proof}
      The proof is inspired by \cite{hu2016control}. \modif{We want to find the virtual control input $\hat U_i$ such that the function $\alpha_i(t,1)$ (right output of this subsystem $i$) tracks the reference signal $\zeta_i$.}
       Let us first introduce the intermediate virtual control input $\hat U_i^{tr}(t)$ such that for all $t>t_0+\frac{1}{\lambda_i^1}$
      \begin{align}
           \hat U_i(t)&= \hat U_i^{tr}(t)-Q_{i,i}P_{v_i}(t,0)\nonumber \\
           &-\int_0^1 (K_i^{uu}(0,y)P_{u_i}(t,y)dy+K_i^{uv}(0,y)P_{v_i}(t,y)dy\nonumber \\
& +Q_{i,i-1}\int_0^{\frac{1}{\lambda_{i-1}^1}}F_{i-1}(1,y)P_{v_i}(t-y,0). \label{Chap_3_control_intermediate}
       \end{align}
       This gives $\alpha_i(t,0)=\hat U_i^{tr}(t-\sum_{j=1}^{i-1}\frac{1}{\lambda_j^1})$.
Due to Lemma~\ref{Chap_3_Lemma_Characteristics} (given in Appendix), the control law $(\hat U_i^{tr}(t))$ defined for all $1\leq j \leq n_i$ by
\begin{align}
    &(\hat U_i^{tr}(t))_{j}=\alpha_i^j(t+\sum_{j=1}^{i-1}\frac{1}{\lambda_j^1},0)=\zeta_i^j(t+\frac{1}{\lambda_i^j}+\sum_{j=1}^{i-1}\frac{1}{\lambda_j^1}) \nonumber \\
    &+\sum_{\ell=j+1}^{n_i}\int_0^{\frac{1}{\lambda_i^j}}(\check G_i(0,\nu))_{j\ell}\zeta_i^\ell(t+\sum_{j=1}^{i-1}\frac{1}{\lambda_j^1}+\nu)d\nu, \label{Chap_3_tracking_law}
\end{align}
guarantees $\alpha_i(t,1)=\zeta_i(t)$ for any $t\geq t_0+\frac{1}{\lambda_i^1}$. \modif{This expression shows that $\alpha_i(t,1)$ corresponds to a flat output~\cite{levine2011advances} that is used for trajectory planning (similarly to~\cite{meurer2009tracking}). The controller $\hat U_i$ can be seen as a flatness-based feedforward tracking controller. Note that Lemma~\ref{Chap_3_Lemma_Characteristics} provides flatness-based parametrization of the PDE state.}
\modif{Let us now consider that $\zeta_i(t) \equiv 0$, and $v_{i+1}(t) \equiv 0$. Then the $i^{\text{th}}$ subsystem with the control law~\eqref{Chap_3_control_intermediate} is now autonomous,  
% It rewrites 
% \begin{align*}
%     \partial_t \alpha_i(t,x)+\Lambda_i^+ \partial_x \alpha_i(t,x)&=G_i(x) \alpha_i(t,x), \\
% \partial_t\beta_i(t,x)-\Lambda_i^-  \partial_x\beta_i(t,x)&=\bar G_i(x)\alpha_i(t,1), 
% \end{align*}
with the boundary conditions $\alpha_i(t,0) = 0,$ and $\beta_i(t,1) =R_{i,i} \alpha_i (t,1).$
As shown in~\cite{auriol2016minimum}, this target system is exponentially stable (and even finite-time stable).}   $\blacksquare$
   \end{proof}
   The fact that we need future values of the functions $u_j$, $v_j$, $\alpha_j$, $\beta_j$ ($j\geq i$) is induced by the presence of the delay $\sum_{j=1}^{i-1} \frac{1}{\lambda_j^1}$ in the virtual control input $\hat U_i(t)$.  Due to the transport delay to go from the left boundary of the $\alpha_i$-PDE ($x=0$, where is located the virtual actuation) to its right boundary ($x=1$, where is defined the output we want to track), we also need future values of the reference signal $\zeta_i$. However, one can verify in the proof of Property~\ref{Chap_3_assum_tracking} that only $(t+\sum_{k=1}^{i}\frac{1}{\lambda_k^1})$-ahead of time values of $\zeta_i$ are required.  Finally, we emphasize that Property~\ref{Chap_3_assum_tracking} does not have to be satisfied for the last subsystem. 
%%%%%%%%%%%%%%%%%%%%%%%%%%%

%%%%%%%%%%%%%%%%%%%%%%%%%%%
\subsection{Input-to-State stability}
As explained in Section~\ref{Chap_3_rec_design}, the control framework we propose recursively stabilizes each subsystem, starting from the last one. %For a given subsystem, the corresponding control input corresponds to the reference signal the upstream subsystem has to track. 
However, to guarantee closed-loop stability of the whole chain, we need the following Input-to-State Stability (ISS) property for each subsystem. 
\begin{property}\label{Chap_3_assum_ISS}
Consider the $i^{th}$ subsystem $(i\in \{1,\hdots,N-1\})$ and consider that Property~\ref{Chap_3_assum_tracking} holds, where $\hat U_i(t)$ is defined by equation~\eqref{Chap_3_control_intermediate}. Then, there exist two constants $\kappa_i>0$ and $\eta_i>0$ such that for all  $t>t_0+\frac{1}{\lambda_i^1}+\frac{1}{\mu_i^1}$, we have 
\begin{align}
       ||(\alpha_i(t,\cdot),\beta_i(t,\cdot))||^2_{L^2} \leq \kappa_i\big(&||(\zeta_i)_{[t]}||^2_{L^2_{\eta_i}}+||(\zeta_i)_{[t]}||^2_{L^2_{-\eta_i}}\nonumber \\
       &+||(v_{i+1}(\cdot,0))_{[t]}||^2_{L^2_{\eta_i}}\big). \label{Chap_3_eq_ISS}
   \end{align}
   \end{property}
   \begin{proof}
   Due to Property~\ref{Chap_3_assum_tracking}, we have for all $t>t_0+ \frac{1}{\lambda_i^1}$, $\alpha_i(t,1)=\zeta_i(t)$. Applying the method of characteristics on equation~\eqref{Chap_3_PDE_beta}, \modif{we can express $\beta(t,x)$ as a delayed function of $\zeta_i(t)$ and $v_{i+1}(t,0)$}. Using the boundary condition~\eqref{Chap_3_bound_beta}, we obtain for all $j \in \{1, \hdots, N\}$, for all $t>t_0+ \frac{1}{\lambda_i^1}+\frac{1}{\mu_i^1}$
   \begin{align}
      & \beta_i^j(t,x)=\sum_{k=1}^{n_i} (R_{i,i})_{jk}\zeta_i(t-\frac{1-x}{\mu_i^j})\nonumber \\
       &+\sum_{k=1}^{m_{i+1}} (R_{i,i+1})_{jk}v^k_{i+1}(t-\frac{1-x}{\mu_i^j},0)-\sum_{k=1}^{n_i}\sum_{\ell=1}^{m_{i+1}}\nonumber\\& \int_0^{\frac{1}{\lambda_i^1}}(R_{i,i})_{jk}(F_i(1,y))_{k\ell}v_{i+1}^\ell(t-y-\frac{1-x}{\mu_i^j},0)dy\nonumber \\
       &+\sum_{k=1}^{n_i}\int_0^{\frac{1-x}{\mu_i^j}}(\bar G_i(x+\mu_i^j\nu))_{jk}\zeta_i^k(t-\nu)d\nu\nonumber \\
       &+\sum_{k=1}^{m_{i+1}}\int_0^{\frac{1-x}{\mu_i^j}}(\bar f_i(x+\mu_i^j\nu))_{jk}v_{i+1}^k(t-\nu,0)d\nu. \label{eq_beta_x}
   \end{align}
   Since the functions $\bar f_i$ and $\bar G_i$ are bounded, straightforward (but tedious), computations give the existence of a constant $K_{\beta_i}>0$ such that 
   \begin{align*}
       ||\beta_i(t,\cdot)||^2_{L^2} \leq K_{\beta_i}\big(||(\zeta_i)_{[t]}||^2_{L^2_{\frac{1}{\mu_i^1}}}+||(v_{i+1}(\cdot,0))_{[t]}||^2_{L^2_{\frac{1}{\mu_i^1}+\frac{1}{\lambda_i^1}}}\big).
   \end{align*}
   Similarly, we can show using Lemma~\ref{Chap_3_Lemma_Characteristics}, that there exists a constant $K_{\alpha_i}>0$ such that $
       ||\alpha_i(t,\cdot)||^2_{L^2} \leq K_{\alpha_i}||(\zeta_i)_{[t]}||^2_{L^2_{-\frac{1}{\lambda_i^1}}}.$ This concludes the proof.$\blacksquare$
   \end{proof}
   The right-hand side of equation~\eqref{Chap_3_eq_ISS} involves past and future values of the functions $\zeta_i$, which is not an issue from a stability perspective. Moreover, due to Property~\ref{Chap_3_assum_ISS},  the finite-time convergence to zero of the functions $\zeta_i$ and $v_{i+1}$ directly implies the finite-time stability of the state $(\alpha_i,\beta_i)$. 
%%%%%%%%%%%%%%%%%%%%%%%%%%%%%%%%%
\subsection{State prediction} \label{Chap_3_Sec_prediction}
The virtual control law given in Property~\ref{Chap_3_assum_tracking} requires the prediction of future values of the functions $\alpha_i$, $\beta_i$, $u_i$, and $v_i$. The following property states that it is possible to design such predictors.
\begin{property}\label{Chap_3_assum_prediction}
Consider the $i^{th}$ subsystem~\eqref{Chap_3_PDE_alpha}-\eqref{Chap_3_bound_beta} $(i\in \{1,\hdots,N\})$ with the virtual input $\hat U_i(t)$ defined in equation~\eqref{Chap_3_virtual_input}. For $t>\max_r\tau_r+\sum_{k=1}^{i-1}\frac{1}{\lambda_k^1}$, for all $x \in [0,1]$, and all $j \in \{i, \hdots, N\}$, it is possible to obtain a $\sum_{k=1}^{i-1}\frac{1}{\lambda_k^1}$-ahead of time prediction of the functions $u_j(t,x)$, $v_j(t,x)$, $\alpha_j(t,x)$, $\beta_j(t,x)$. More precisely there exist predictor functions $P_{u_j}$, $P_{v_j}$, $P_{\alpha_j}$, and $P_{\beta_j}$ \modif{that only depends on past values of the state}, such that for all $t>\max_r\tau_r+\sum_{k=1}^{i-1}\frac{1}{\lambda_k^1}$, for all $x \in [0,1]$,
$P_{u_j}(t,x)=u_j(t+\sum_{k=1}^{i-1}\frac{1}{\lambda^1_k},x),$
$P_{v_j}(t,x)=v_j(t+\sum_{k=1}^{i-1}\frac{1}{\lambda^1_k},x),$ 
$P_{\alpha_j}(t,x)=\alpha_j(t+\sum_{k=1}^{i-1}\frac{1}{\lambda^1_k},x),$
$P_{\beta_j}(t,x)=\beta_j(t+\sum_{k=1}^{i-1}\frac{1}{\lambda^1_k},x).$
\end{property}

 \begin{proof} 
 Consider the $i^{th}$ subsystem~\eqref{Chap_3_PDE_alpha}-\eqref{Chap_3_bound_beta} $(i\in \{1,\hdots,N\})$ with the virtual input $\hat U_i(t)$ defined in equation~\eqref{Chap_3_virtual_input}. Consider $j \in \{i, \hdots, N\}$. For $t>\max_r\tau_r+\sum_{k=1}^{i-1}\frac{1}{\lambda_k^1}$, we will design predictors for the functions $v_j(t,0)$, $\alpha_j(t,1)$ and $\alpha_j(t,0)$. From these predictors, it will be possible to predict the functions $u_j(t,x)$, $v_j(t,x)$, $\alpha_j(t,x)$, $\beta_j(t,x)$ ($x \in [0,1]$). 
Due to the backstepping transformation~\eqref{Chap_3_eq_integral_time_terms_inverse_2}, we have
\begin{align*}
    &v_j(t,0)=\beta_j(t,0)+\int_0^{\frac{1}{\lambda_j^1}} H_i^\beta(0,y)v_{j+1}(t-y,0)dy\\
    &+\int_0^1L_j^{\beta\alpha}(0,y)\alpha_j (t,y)+L_j^{\beta\beta}(0,y)\beta_j(t,y)dy.
\end{align*}
Combining equation~\eqref{Chap_3_eq_alpha_delay} and equation
\eqref{eq_beta_x} (where $\zeta_i=\alpha_i(t,1)$), we obtain, for all $1\leq k \leq m_j$,
\begin{align}
    v_j^k(t,0)&=\sum_{\ell=1}^{n_j}(R_{j,j})_{k\ell}\alpha_j^\ell(t-\frac{1}{\mu_j^k},1)+\sum_{\ell=1}^{m_{j+1}}(R_{j,j+1})_{k\ell}\nonumber \\
    & v_{j+1}^\ell(t-\frac{1}{\mu_j^k},0)+\sum_{\ell=1}^{n_j}\int_0^{\tau_j} (g_j^1)_{k\ell}(\nu)  \alpha_j^\ell(t-\nu,0)d\nu\nonumber \\
    &+\sum_{\ell=1}^{m_{j+1}}\int_0^{\tau_j} (g_j^2)_{k\ell}(\nu) v_{j+1}^\ell(t-\nu,0)d\nu, \label{Chap_3_eq_neutral_v}
\end{align}
where $g_i^1$ and $g_i^2$ are piecewise continuous functions. We do not give their explicit expression for the sake of concision. We recall that by convention $v_{N+1}(t,0) \equiv 0$. Consider now equation~\eqref{Chap_3_eq_alpha_i_0_rewritten}. We can substitute the terms $u_i(t,\cdot)$ and $v_i(t,\cdot)$ that appear in the right-hand side of this equation by their expressions as functions of $\alpha_i(t,\cdot)$, $\beta_i(t,\cdot)$ and $v_{i+1}(t,0)$ using the inverse transformations~\eqref{Chap_3_eq_integral_time_terms_inverse}-\eqref{Chap_3_eq_integral_time_terms_inverse_2}. Then, applying the method of characteristics (see~\cite{auriol2019explicit,auriol2021stabilization} and equations~\eqref{Chap_3_eq_alpha_delay} and~\eqref{eq_beta_x}), \modif{we can express $\alpha(t,0)$ as a delayed function of $\alpha(t,0), \alpha(t,1)$ and $v(t,0)$}. We obtain for $j>i$,
\begin{align}
    \alpha_j^k(t,0)&=\sum_{\ell=1}^{m_j}(Q_{j,j})_{k\ell}v_j^\ell(t,0)+\sum_{\ell=1}^{n_{j-1}}(Q_{j,j-1})_{k\ell}\alpha^\ell_{j-1}(t,1)\nonumber \\
    &+\sum_{\ell=1}^{n_j}\int_0^{\tau_j} (h_i^2)_{k\ell}(\nu) \alpha_j^\ell(t-\nu,0)d\nu \nonumber \\
&+\sum_{\ell=1}^{m_{j+1}}\int_0^{\tau_j} (h_j^2)_{k\ell}(\nu) v_{j+1}^\ell(t-\nu,0)d\nu\nonumber \\
&+\sum_{\ell=1}^{m_{j}}\int_0^{\tau_{j-1}} (h_j^3)_{k\ell}(\nu) v_{j}^\ell(t-\nu,0)d\nu,\label{Chap_3_eq_neutral_alpha}
\end{align}
where $h_j^1$, $h_j^2$, and $h_j^3$  are piecewise continuous functions. We do not give their explicit expression for the sake of concision.  If $j=i$, the term $\sum_{\ell=1}^{n_{i-1}}(Q_{i,i-1})_{k\ell}\alpha^\ell_{i-1}(t,1)$ is replaced by $\sum_{\ell=1}^{n_{j-1}}(Q_{i,i-1})_{k\ell}(\hat U_i(t-\sum_{r=1}^{i-1}\frac{1}{\lambda_r^1}))_\ell$. Inspired by~\cite{bekiaris2014simultaneous,bresch2016prediction,auriol2020CDCunderactuation,auriol2022robust}, we respectively define for $t \geq \max_r\tau_r+\sum_{r=1}^{i-1}\frac{1}{\lambda_r^1}$, $k \in \{1,\hdots,n_j\}$, $\ell \in \{1,\hdots,m_j\}$  and $s \in [t-\max_r\tau_r-\sum_{r=1}^{i-1}\frac{1}{\lambda_r^1},t]$, the functions $P_{\alpha^0_j}^k(t,s)$, $P_{\alpha^1_j}^k(t,s)$, and $P_{v^0_j}^\ell(t,s)$ as the \textbf{state predictions} of $\alpha_j^k(t,0)$, $\alpha_j^k(t,1)$, and $v_\ell^k(t,0)$ ahead a time $\sum_{r=1}^{i-1}\frac{1}{\lambda_r^1}$. They are \emph{explicitly} defined by equations~\eqref{Chap_3_eq:def-P_alpha_0}-\eqref{Chap_3_eq:def-P_v_0} with the convention $\sum_{\ell=1}^{n_{i-1}}(Q_{i,i-1})_{kq}P_{\alpha_{i-1}^1}^q(t,s)=\hat U^q_i(s)$.  \modif{We write the predictors as functions of two arguments to emphasize that the predictions should be computed by incorporating delayed states
available at time $t$ to improve its robustness in practice. It should be noticed that the expressions of these predictors are causal as they only depend on past values of the functions $\alpha_j^k(t,0)$, $\alpha_j^k(t,1)$ and $v_j^k(t,0)$}. 
\begin{table*}
\begin{align}
	& P^k_{\alpha_j^0}(t,s) =\left\{ 
		\begin{aligned}
			& \alpha^k_j(s+\sum_{r=1}^{i-1}\frac{1}{\lambda_r^1},0) \, \hspace{0.25cm} \mbox{if } s \in [t-\max_r\tau_r-\sum_{r=1}^{i-1}\frac{1}{\lambda^1_r},t-\sum_{r=1}^{i-1}\frac{1}{\lambda_r^1}] \\
			& \sum_{q=1}^{m_j}(Q_{j,j})_{kq}P_{v_j^0}^q(t,s)+\sum_{q=1}^{n_{j-1}}(Q_{j,j-1})_{kq}P_{\alpha_{j-1}^1}^q(t,s)+\sum_{q=1}^{n_j}\int_0^{\tau_j} (h_i^1)_{kq}(\nu) P_{\alpha_j^0}^q(t,s-\nu)d\nu  \\
&+\sum_{q=1}^{m_{j+1}}\int_0^{\tau_j} (h_j^2)_{kq}(\nu) P_{v_{j+1}^0}^q(t,s-\nu)d\nu+\sum_{q=1}^{m_{j}}\int_0^{\tau_{j-1}} (h_j^3)_{kq}(\nu) P_{v_j^0}^q(t,s-\nu)d\nu,~\mbox{ otherwise,}
		\end{aligned}
	\right.\label{Chap_3_eq:def-P_alpha_0}
\\
	& P^k_{\alpha_j^1}(t,s) =\left\{ 
		\begin{aligned}
		&\alpha^k_j(s+\sum_{r=1}^{i-1}\frac{1}{\lambda_r^1},1) \, \hspace{0.25cm} \mbox{if } s \in [t-\max_r\tau_r-\sum_{r=1}^{i-1}\frac{1}{\lambda^1_r},t-\sum_{r=1}^{i-1}\frac{1}{\lambda^1_r}] \\
			& P^k_{\alpha_j^0}(t,s-\frac{1}{\lambda_j^k})+\sum_{q=k+1}^{n_j}\int_0^{\frac{1}{\lambda_j^k}}(\tilde G_j(1-\lambda_j^k\nu))_{kq}P^q_{\alpha_j^0}(t,s-\nu)d\nu,~\mbox{ otherwise,}
		\end{aligned}
	\right. \label{Chap_3_eq:def-P_alpha_1}
\\
	& P^\ell_{v_j^0}(t,s) =\left\{ 
		\begin{aligned}
			& v_j^\ell(s+\sum_{r=1}^{i-1}\frac{1}{\lambda_r^1},0) \, \hspace{0.25cm} \mbox{if } s \in [t-\max_r\tau_r-\sum_{r=1}^{i-1}\frac{1}{\lambda_r^1},t-\sum_{r=1}^{i-1}\frac{1}{\lambda^1_r}] \\
			& \sum_{q=1}^{n_j}(R_{j,j})_{\ell q}P_{\alpha_j^1}^q(t,s-\frac{1}{\mu_j^\ell})+\sum_{q=1}^{m_{j+1}}(R_{j,j+1})_{\ell q}P_{v_{j+1}^0}^q(t,s-\frac{1}{\mu_j^\ell})+\sum_{q=1}^{n_j}\int_0^{\tau_j} (g_j^1)_{\ell q}(\nu) P_{\alpha_j^1}^q(t,s-\nu)d\nu  \\
    & +\sum_{q=1}^{m_{j+1}}\int_0^{\tau_j} (g_j^2)_{\ell q}(\nu) P_{v_{j+1}^0}^q(t,s-\nu)d\nu,~\mbox{ otherwise,}
		\end{aligned}
	\right.\label{Chap_3_eq:def-P_v_0}
\end{align}
\end{table*}
 % Though the definitions~\eqref{Chap_3_eq:def-P_alpha_0}-\eqref{Chap_3_eq:def-P_v_0} are implicit, through integral relations of Volterra type, the predictors are well-defined and unique. 

% From these definitions, we immediately have
% \begin{align*}
% &P^k_{\alpha^0_j}(t,s)= \alpha^k_j(s+\sum_{r=1}^{i-1}\frac{1}{\lambda_r^1},0),~\text{$s \in [t-\max_r \tau_r-\sum_{r=1}^{i-1}\frac{1}{\lambda_r^1},t]$}, \\
% &P^k_{\alpha_j^1}(t,s)=\alpha_j^k(s+\sum_{r=1}^{i-1}\frac{1}{\lambda_r^1},1),~\text{$s \in[t-\max_r \tau_r-\sum_{r=1}^{i-1}\frac{1}{\lambda_r^1},t]$},\\
% &P^\ell_{v_j^0}(t,s)=v^\ell_j(s+\sum_{r=1}^{i-1}\frac{1}{\lambda_r^1},0),~\text{$s\in [t-\max_r \tau_r-\sum_{r=1}^{i-1}\frac{1}{\lambda_r^1},t]$}.
% \end{align*} 
From the predictors~\eqref{Chap_3_eq:def-P_alpha_0}-\eqref{Chap_3_eq:def-P_v_0}, it is possible to apply equations~\eqref{eq_beta_x}  and equation~\eqref{Chap_3_eq_alpha_delay}, to obtain the corresponding state predictions for the states $\alpha_j(t,x)$ and $\beta_j(t,x)$. Finally, using the transformations~\eqref{Chap_3_eq_integral_time_terms_inverse}-\eqref{Chap_3_eq_integral_time_terms_inverse_2}, we obtain the predictions of the state $u_j(t,x)$ and $v_j(t,x)$. $\blacksquare$
\end{proof}
Note that the definitions of the predictors implicitly depend on the initial subsystem $i$ we consider. Indeed, the different time horizons depend on the parameter $i$. We chose to omit this dependency as we believe the notations are sufficiently heavy.
\modif{The definitions of the predictors rely on a time-delay representation (inspired from~\cite{auriol2019explicit}) of the interconnected system~\eqref{Chap_3_eq_init_PDE_u}-\eqref{Chap_3_eq_init_bound_v}. In this context, the functions $\alpha_i(t,1)$, $\alpha_i(t,0)$ and $v_i(t,0)$ can be seen as a (quasi) flatness-based parametrization of~\eqref{Chap_3_eq_init_PDE_u}-\eqref{Chap_3_eq_init_bound_v}. Similar parameterizations have been used in the literature for flatness-based open-loop design, controllability analysis, and closed-loop design (see, e.g. \cite{woittennek2013flatness}  or~\cite{petit2001flatness}).}
   %%%%%%%%%%%%%%%%%%%%%%%%%
%%%%%%%%%%%%%%%%%%%%%%%%%%%%%%%%%%%%%
\subsection{Recursive state-feedback stabilization}\label{Sec_recursive_State_feedback}
We now have all the tools to apply our recursive dynamics interconnection framework
\begin{theorem}\label{Chap_3_th_stab_rec}
    For $i \in \{1, \hdots, N\}$, and for $t>\sum_{j=1}^N 2\tau_j$, define the sequences of functions $\zeta_i$, with $\zeta_N(t)=0$ and for all $i<N$
    \begin{align}
    &\zeta_i(t)=(Q^T_{i+1,i}(Q_{i+1,i}Q^T_{i+1,i})^{-1})\hat U_{i+1}(t-\sum_{j=1}^{i}\frac{1}{\lambda_j^1}), \label{Chap_3_eq_def_zeta_i}
    \end{align}
    where the functions $\hat U_i$ are defined by equations~\eqref{Chap_3_control_intermediate}-\eqref{Chap_3_tracking_law}, where the different predictors are given in Property~\ref{Chap_3_assum_prediction} (using the function $\hat U_i$ and equations~\eqref{Chap_3_eq:def-P_alpha_0}-\eqref{Chap_3_eq:def-P_v_0}). Then, the interconnected system~\eqref{Chap_3_eq_init_PDE_u}-\eqref{Chap_3_eq_init_bound_v} with the control law $U(t)=\hat U_1(t)$ is exponentially stable. Moreover, the equilibrium is reached in finite time.
\end{theorem}
\begin{proof}
First observe that the matrices $Q_{i+1,i}^T(Q^T_{i+1,i}Q^T_{i+1,i})^{-1}$ are well defined due to Assumption~\ref{Chap_3_ass-control}. The quantity $\zeta_i(t+\sum_{j=1}^i \frac{1}{\lambda_i^j})$ that appears in the proof of Property~\ref{Chap_3_assum_tracking} can be explicitly computed from $\hat U_{i+1}(t)$. Then, the sequences $\zeta_i$ and $\hat U_i$ are well defined (since equations~\eqref{Chap_3_eq:def-P_alpha_0}-\eqref{Chap_3_eq:def-P_v_0} are always well defined). Consequently, the control input $U(t)$ is well-defined and causal. 

Next, we briefly show that the closed-loop system~\eqref{Chap_3_eq_init_PDE_u}-\eqref{Chap_3_eq_init_bound_v} with the control input $U(t)$ is well-posed. This can be done either by considering the admissibility of the control operator~\cite{coron2016stabilization} (the control law is continuous in time), or by adjusting the proof of Theorem~\cite[Theorem A.1]{bastin2016stability} (that is based on Lumer-Philipps theorem). Indeed, the different components of the proposed control input $U(t)$ (including the predictors) can be expressed as delayed values of the boundary states of the system (as $v_i(t,0)$) or delayed values of themselves. Such delayed values, could then be expressed using PDEs (after tedious computations). %The 'if' condition in the definition of the predictor  Such delayed values, could then be expressed using PDEs. For instance, consider the predictor $P^k_{\alpha_j^1}$ defined by~\eqref{Chap_3_eq:def-P_alpha_1} the term $\int_0^{\tau_j} h_i(\nu)P^q_{\alpha_j^0}(t,s-\nu)$ could be express as $w(t,1)$, where $w$ is the solution of the PDE
%\begin{align*}
   % \partial_t w -\frac{1}{\tau_i} 
%\end{align*}Therefore, these functions can be rewritten noticing that 

We now need to prove that the proposed control law stabilizes the system. To ease the computations, the parameter $T$ (that will be  overloaded in the rest of the proof) denotes a finite time large enough to guarantee that the different predictors and tracking controllers are well-defined. %Note that for a subsystem $i$, the ‘‘predictors" defined in the proof of Property~\ref{Chap_3_assum_prediction} correspond to exact predictions of the different states only if the subsystem is subject to the virtual input $\hat U_i(t)$, i.e., only if $Q_{i,i+1}\alpha_{i-1}(t+\sum_{j=1}^{i-1}\frac{1}{\lambda_j^1},1)=\hat U_i(t)$. 
Consider the first subsystem ($i=1$) with the control law $U(t)=\hat U_1(t)$. For $i=1$, equation~\eqref{Chap_3_control_intermediate} and equation~\eqref{Chap_3_tracking_law} do not require any state predictions but can be computed using current values of the different functions. Then, using Property~\ref{Chap_3_assum_tracking}, we have that $\alpha_1(t,1)=\zeta_1(t)$ for $t>T$. Consequently, $Q_{2,1}\alpha_1(t,1)=\hat U_2(t-\frac{1}{\lambda_1^1}).$ Therefore the functions defined through~equations~\eqref{Chap_3_eq:def-P_alpha_0}-\eqref{Chap_3_eq:def-P_v_0} corresponds to exact $\frac{1}{\lambda_1^1}$-ahead of time predictions of the real states. Thus, Property~\ref{Chap_3_assum_tracking} implies that $\alpha_2(t,1)=\zeta_2(t)$ after a finite time $T$. Iterating the procedure, we obtain that after a finite time $T$, for all $i\in \{1, \hdots, N\}$,  $\alpha_i(t,1)=\zeta_i$ %and consequently $Q_{i,i-1}\alpha_{i-1}(t,0)=\hat U_i(t-\sum_{j=1}^{i-1}\frac{1}{\lambda_j^1}).$ 
Consider now the last subsystem ($i=N$). Since $\alpha_N(t,1)=\zeta_N=0$, the functions $\alpha_N(t,x)$ and $\beta_N(t,x)$ converge to zero in finite time. Applying Property~\ref{Chap_3_assum_ISS}, we obtain the convergence to zero of the functions $\alpha_{N-1}(t,x)$ and $\beta_{N-1}(t,x)$ in finite time. Iterating the procedure, all the states $(\alpha_i,\beta_i)$ converge to zero in finite time. Using the inverse backstepping transformations~\eqref{Chap_3_eq_integral_time_terms_inverse}-\eqref{Chap_3_eq_integral_time_terms_inverse_2}, we obtain that the system~\eqref{Chap_3_eq_init_PDE_u}-\eqref{Chap_3_eq_init_bound_v} reaches its equilibrium in finite-time. The well-posedness of the closed-loop system implies its exponential stability.
%Let us now show that the closed-loop system is exponentially stable.
%To prove that the system is exponentially stable, we can combine the well-posedness of the system %with Theorem~\ref{theorem_equiv_norm}
%and Property~\ref{Chap_3_assum_ISS}. We omit the complete proof (which requires tedious yet classical developments) for the sake of concision. {\color{red} Reprendre le côté conv. exp.} 
$\blacksquare$
%From the definitions of the predictors ~\eqref{Chap_3_eq:def-P_alpha_0}-\eqref{Chap_3_eq:def-P_v_0}, their $L^2$ equations~\eqref{Chap_3_eq_alpha_delay}-\eqref{Chap_3_eq_neutral_alpha}
%Considering the extended state composed of $(\alpha_i(t,0),\alpha_i(t,1),v_i(t,0))$ the definitions of the predictors (equations~\eqref{Chap_3_eq:def-P_alpha_0}-\eqref{Chap_3_eq:def-P_v_0}) and of the virtual control laws $\hat U_i$,  %Borner les prédicteurs, puis les lois de commande...
\end{proof}
%%%%%%%%%%%%%%%%%%%%%%%%%%%%%%%%%
One major advantage of the proposed framework and of the recursive design presented in Theorem~\ref{Chap_3_th_stab_rec} is that it can easily be extended to different classes of subsystems (as ODEs, for instance), as long as it is possible to derive analogous properties to  Property~\ref{Chap_3_assum_tracking}, Property~\ref{Chap_3_assum_ISS} and Property~\ref{Chap_3_assum_prediction}.

\begin{remark} \label{Chap_3_Remark_delay}
The state-feedback controller designed in Theorem~\ref{Chap_3_th_stab_rec} can be easily extended to the case of a delayed control input. Indeed, similarly to what has been done in the case of ODEs~\cite{krstic2008boundary}, one needs to consider an additional upstream subsystem corresponding to a pure transport equation, thus inducing a delay corresponding to the input delay.
\end{remark}

\modif{
\subsection{A remark on robustness and computational aspects} \label{Sec_filter}
%One must be aware that Theorem~\ref{Chap_3_th_stab_rec} completely neglects the robustness aspects of the system.
Although the controller designed in Theorem~\ref{Chap_3_th_stab_rec} fulfills the control objective and stabilizes the system~\eqref{Chap_3_eq_init_PDE_u}-\eqref{Chap_3_eq_init_bound_v}, it presents several drawbacks that could impact its implementation:
\begin{itemize}[label=\textbullet, labelindent=0pt, leftmargin=*]
    \item First, the proposed approach is based on an exact reconstruction of the state using state predictors. The computation of the predictors can be time-consuming (as the numerical complexity increases with the number of subsystems) and troublesome. Indeed, implicit expressions of prediction-based feedback sometimes lead to burdensome numerical procedures and induce poor robustness margins, as shown in~\cite{karafyllis2017predictor,mondie2003finite} for the ODE case. For fully actuated integral difference equations (but with a known delay in the control input), it has been shown in~\cite{auriol2022explicit} that it is possible to obtain explicit expressions of such predictors and consequently improve the efficiency and robustness of the proposed control design. We believe it may be possible to obtain such an explicit expression, even if the results from~\cite{auriol2022explicit} do not directly apply due to the underactuated configuration. Overall, it is essential to envision model reduction strategies in order to apply the proposed control law to real test cases. 
    \item Then, the proposed approach consists of recursively canceling all the boundary reflection terms for each subsystem to track the virtual input of the downstream subsystem. This may lead to vanishing robustness margins, as shown in~\cite{auriol2017delay}. Unfortunately, the robustification procedure proposed in~\cite{auriol2023robustification} cannot be directly applied since, due to the tracking part, our control law does not fit in the framework of the theorems given in~\cite{auriol2023robustification}. However, using the cascade structure of the problem, it should be possible (under Assumption~\ref{Ass_robustness}) to extend the results from~\cite{auriol2023robustification} to the system~\eqref{Chap_3_eq_init_PDE_u}-\eqref{Chap_3_eq_init_bound_v}. This solution was successfully tested in simulations.
\end{itemize}
}

%%%%%%%%%%%%%%%%%%%%%%%%%%%%%%%%%%%%%%%%%%%%%%%%%%%%%%%%%%%%%%%%%%%%%%%%%%%%%
\section{State estimation and output-feedback stabilization} \label{Chap_3_Sec_State_Estimation}

To design the recursive stabilizing controller we presented in Section~\ref{Chap_3_Sec_SF}, we need the knowledge of the states $u_i(t,x)$ and $v_i(t,x)$ all over the spatial domain $[0,1]$. Since the available measurement corresponds to $u_N(t,1)$, we must design a state observer. In this section, we show how to easily obtain estimated \textbf{delayed} values of these states. Adjusting the predictors introduced in Section~\ref{Chap_3_Sec_prediction}, it is then possible to reconstruct the desired states.

\subsection{Delayed interconnection}
 Inspired by~\cite{karafyllis2017predictor}, we consider a delayed version of the interconnected system~\eqref{Chap_3_eq_init_PDE_u}-\eqref{Chap_3_eq_init_bound_v}. Let us consider~$\tau>\sum_{j=1}^N\frac{1}{\lambda_j^1}>0$ a fixed, known delay. We define the \emph{$\tau$-delay operator~$\bar \cdot$}, such that for all functions $\gamma$ defined on $[0,+\infty)$, $\forall t>\tau, \bar \gamma(t)=\gamma(t-\tau)$. Using this operator, we can obtain the $\tau$-delayed version of system~\eqref{Chap_3_eq_init_PDE_u}-\eqref{Chap_3_eq_init_bound_v}. For all~$t\geq \tau$, we have:
\begin{align}
\partial_t \overline{u}_i(t,x)+\Lambda^+_i \partial_x \overline{u}_i(t,x)&=\Sigma^{++}_i(x) \overline{u}_i+\Sigma^{+-}_i(x) \overline{v}_i, \label{Chap_3_eq_delayed_PDE_u}\\
\partial_t  \overline{v}_i(t,x)-\Lambda^-_i  \partial_x \overline{v}_i(t,x)&=\Sigma^{-+}_i(x) \overline{u}_i+\Sigma^{--}_i(x) \overline{v}_i,
\end{align}
with the boundary conditions:
\begin{align}
\overline{u}_i(t,0) &= Q_{i,i}\overline{v}_i(t,0)+Q_{i,i-1}\overline{u}_{i-1}(t,1),
\label{Chap_3_eq_init_bound_u_delay}\\
\overline{v}_i(t,1) &=R_{i,i} \overline{u}_i (t,1) +R_{i,i+1} \overline{v}_{i+1}(t,0), \label{Chap_3_eq_init_bound_v_delay}
\end{align} 
where we still use the convention that $Q_{1,0}\bar u_0(t,0)=U(t-\tau)$ and $R_{N,N+1}=0$. 
The available measurement is now given as~$\bar y(t)=y(t-\tau)$. It implies that we know~$\tau$-ahead future values of the function~$\bar y$(t).
%Let us introduce the sequence $t_i$ defined by 
%\begin{align}
 %   t_i=\sum_{j=i+1}^N\frac{1}{\lambda^1_j}.
%\end{align}
%For a given subsystem~$i \in \{1,\hdots,N\}$, we define the virtual measurement~$\overline{y}_i$ as:
%$$
%\bar y_i(t)=u_i(t-t_i,1)=\bar u_i(t+\tau-t_i,1).
%$$
%This definition is causal as it only requires past values of the function~$\bar u_i(\cdot,1)$. Note that~$\bar y_N=\bar y$ is known on time interval~$[t,t+\tau]$.
Using the backstepping transformations~\eqref{Chap_3_eq_integral_time_terms}-\eqref{Chap_3_eq_integral_time_terms_2}, we can define the states $\bar \alpha_i(t,x)$ and $\bar \beta_i(t,x)$. They are solutions of~\eqref{Chap_3_PDE_alpha}-\eqref{Chap_3_bound_beta} with a $\tau$-delayed control input. 

\subsection{Estimation of the delayed states }
%{\color{red}This section shows how to estimate the delayed state $\bar u_i(t,x)$ and $\bar v_i(t,x)$  from the available measurements. We first show how to estimate the boundary functions $\bar \alpha_i(t,1)$ and $\bar v_i(t,0)$. These estimations will then be used to reconstruct the states $\bar u_i$ and $\bar v_i$. More precisely, we have the following property.}
We now  estimate the delayed state $\bar u_i(t,x)$ and $\bar v_i(t,x)$  from the available measurements.
\begin{lemma} \label{Chap_3_lemma:observability} %Define the sequence t_i=\sum_{j=i+1}^N\frac{1}{\lambda^1_j}
%Consider the $i^{th}$ subsystem~$i \in \{2,\hdots,N\}$, $t>0$ and assume that the functions $\alpha_i(\nu,1)$ and $\bar v_{i+1}(\nu,0)$ are known for all $\nu \in [t,t+\sum_{j=1}^i\frac{1}{\lambda^1_j}]$. Then, 
For all $i \in \{1,\hdots,N\}$, we can design exact state estimators $\hat \alpha_{i}(\cdot,1)$ and $\hat v_{i+1}(\cdot,0)$ that causally depend on the measurement~$y(t)$ such that for all $\nu \in [t,t+\sum_{j=1}^{i}\frac{1}{\lambda^1_j}]$, $\hat \alpha_{i}(t+\nu,1)=\bar \alpha_{i}(t+\nu,1)$ and $\hat v_{i+1}(t+\nu,0)=\bar v_{i+1}(t+\nu,0)$.
\end{lemma}
\begin{proof}
The proof relies on an induction argument. Lemma~\ref{Chap_3_lemma:observability} obviously holds for $i=N$ with $\hat \alpha_N(t,1)=\bar y(t)$ and $\hat v_{N+1}(t,0)=0$. Let us now consider the $i^{th}$ subsystem~$i \in \{2,\hdots,N\}$, $t>0$ and assume that we can design exact state estimations $\hat \alpha_i(\nu,1)$ and $\hat  v_{i+1}(\nu,0)$ that causally depend on the measurement $y(t)$ such that for all $\nu \in [t,t+\sum_{j=1}^i\frac{1}{\lambda^1_j}]$, $\hat \alpha_i(t+\nu,1)=\bar \alpha_{i}(t+\nu,1)$ and $\hat  v_{i+1}(t+\nu,0)=\bar v_{i+1}(t+\nu,0)$.  From equation~\eqref{Chap_3_eq_alpha_delay_1}, we can define the intermediate estimator $\hat \alpha_i(t,0)$ such that  for all $1\leq k \leq n_{i}$ and all $t>0$
\begin{align}
    \hat \alpha^k_{i}(t,0) &= \hat  \alpha^k_{i}(t+\frac{1}{\lambda_{i}^k},1)\nonumber \\
    &-\sum_{\ell=k+1}^{n_{i}}\int_0^{\frac{1}{\lambda_{i}^k}}(\check G_{i}(\nu))_{k\ell} \hat   \alpha_{i}^\ell(t+\nu,1)d\nu. \label{Chap_3_eq_alpha_delay_est}
\end{align}
We immediately obtain that for all $\nu \in [t,t+\sum_{j=1}^{i-1}\frac{1}{\lambda^1_j}]$, $\hat \alpha_i(t+\nu,0)=\bar \alpha_{i}(t+\nu,0)$. We now define the function $\hat v_i(t,0)$, such that for all $1\leq k \leq m_i$
\begin{align}
    &\hat v_i^k(t,0)=\sum_{\ell=1}^{n_i}(R_{i,i})_{k\ell}\hat \alpha_j^\ell(t-\frac{1}{\mu_i^k},1)+\sum_{\ell=1}^{m_{i+1}}(R_{i,i+1})_{k\ell}\nonumber \\
    &\hat v_{i+1}^\ell(t-\frac{1}{\mu_i^k},0) +\sum_{\ell=1}^{n_i}\int_0^{\tau_i} (g_i^1)_{k\ell}(\nu)  \hat\alpha_i^\ell(t-\nu,0)d\nu\nonumber \\
    &+\sum_{\ell=1}^{m_{i+1}}\int_0^{\tau_i} (g_i^2)_{k\ell}(\nu) \hat v_{i+1}^\ell(t-\nu,0)d\nu. \label{Chap_3_eq_neutral_v_est}
\end{align}
We have that for all $\nu \in [t,t+\sum_{j=1}^{i-1}\frac{1}{\lambda^1_j}]$, $\hat v_i(t+\nu,0)=\bar v_{i}(t+\nu,0)$ due to equation~\eqref{Chap_3_eq_neutral_v}.
Finally, combining  Assumption~\ref{Chap_3_ass-obs} and equation~\eqref{Chap_3_eq_neutral_alpha}, we can obtain the desired estimations of $\bar \alpha_{i-1}(t,1)$. The different estimators are causal as they only require past values of the function~$y$. This concludes the proof. $\blacksquare$
\end{proof}
From Lemma~\ref{Chap_3_lemma:observability}, we obtain the following property
\begin{property} \label{Chap_3_Prop_estim}
    For all $i \in \{1,\hdots,N\}$, we can design exact state estimators $\hat u(t,x)$ and $\hat v(t,x)$ that causally depend on the measurement~$y(t)$ such that for all $t>0$, and all $x\in [0,1]$, $\hat u(t,x)=\bar u(t,x)$ and $\hat v(t,x)=\bar v(t,x)$.
\end{property}
\begin{proof}
   Combining the state estimations given in Lemma~\ref{Chap_3_lemma:observability} with the method of characteristics, it is possible to estimate the state $\bar \alpha_i$ and $\bar \beta_i$. Then, we can compute the estimators $\hat u(t,x)$ and $\hat v(t,x)$ using the inverse transformations~\eqref{Chap_3_eq_integral_time_terms_inverse}-\eqref{Chap_3_eq_integral_time_terms_inverse_2}. We do not give the explicit expression of these state estimators for the sake of concision.
\end{proof}

%%%%%%%%%%%%%%%%%%%%%%%%%%%%%%%%%%%%%%%%%%%%%%%%%%%%%%%%%%%%%%%%%%%%%%%%%%%%%%%%%%%%%%%%%%%%%%
\subsection{Stabilizing output-feedback controller} \label{Chap_3_Sec_Output_Feedback}
We have designed in Property~\ref{Chap_3_Prop_estim} a state-observer that provides a real-time exact estimation of the delayed states $(\bar u_i, \bar v_i )$. This state-observer can be combined with the state-feedback controller designed in Theorem~\ref{Chap_3_th_stab_rec} to obtain an output-feedback stabilizing controller. Indeed, combining Remark~\ref{Chap_3_Remark_delay} and Theorem~\ref{Chap_3_th_stab_rec}, we can design a state-feedback controller for the delayed system~\eqref{Chap_3_eq_delayed_PDE_u}-\eqref{Chap_3_eq_init_bound_v_delay}. This state feedback controller requires the knowledge of the delayed states $(\bar u_i, \bar v_i )$, provided by Property~\ref{Chap_3_Prop_estim}. Therefore, we can obtain a stabilizing output-feedback controller for the delayed system~\eqref{Chap_3_eq_delayed_PDE_u}-\eqref{Chap_3_eq_init_bound_v_delay}. The exponential stability of the delayed system~\eqref{Chap_3_eq_delayed_PDE_u}-\eqref{Chap_3_eq_init_bound_v_delay} implies the exponential stability of the original system \eqref{Chap_3_eq_init_PDE_u}-\eqref{Chap_3_eq_init_bound_v}.
%%%%%%%%%%%%%%%%%%%%%%%%%%%%%%%%%%%%%%%%%%%%%%%%%%%%%%%%%%%%%%%%

\section{Simulation results} \label{Section_Simulations}

We now illustrate our results in simulations using Matlab. The PDE system is simulated using \modif{a high-resolution explicit scheme similar to the one used in~\cite{deutscher2023backstepping} with 101 spatial discretization points}. The algorithm we use to compute the different kernels is the following. Using the method of characteristics, we write the integral equations associated to the kernel PDE-systems. These integral equations are solved using a fixed-point algorithm. The predictor is implemented using a backward Euler approximation of the integral. The numerical values used are given below
\begin{align*}
    &\Lambda_1^+=\begin{pmatrix}
        1 &0 \\ 0 & 1.4
    \end{pmatrix},~\Lambda_1^-=1.2, ~\Sigma^{++}_1=\begin{pmatrix}
        0 &0.4 \\ 0.4 & 0
    \end{pmatrix}, \\
    &\Sigma^{--}_1=0,~\Sigma^{-+}_1(x)=\begin{pmatrix}
        0.2 &0.9+0.1\sin(x)
    \end{pmatrix},\\
    &\Sigma^{+-}_1=\begin{pmatrix}
        0.8 \\ 0.8+0.1\cos(x)
    \end{pmatrix},~\Lambda_2^+=\begin{pmatrix}
        1 &0 \\ 0 & 1.5
    \end{pmatrix}, \\
    & \Lambda_2^-=1,~\Sigma^{++}_2=\begin{pmatrix}
        0 &1 \\ 0.5 & 0
    \end{pmatrix},~\Sigma^{-+}_2=\begin{pmatrix}
        0.2 &0
    \end{pmatrix}, \\
    &\Sigma^{--}_2=0,~\Sigma^{+-}_2=\begin{pmatrix}
       1 \\ 0
    \end{pmatrix},~\Lambda_3^+=\begin{pmatrix}
        1.1 &0 \\ 0 & 1.6
    \end{pmatrix}, \\
    & ~\Lambda_3^-=0.8,~\Sigma^{++}_3=\begin{pmatrix}
        0 &1 \\ 0.5 & 0
    \end{pmatrix},~\Sigma^{+-}_3=\begin{pmatrix}
        0.3 \\ 0
    \end{pmatrix}, \\
    &\Sigma^{--}_3=0,~\Sigma^{-+}_3=\begin{pmatrix}
        -0.5 &0.2
    \end{pmatrix}, \\
    &Q_{11}=\begin{pmatrix}
        0.1 \\ 0.2
    \end{pmatrix},~R_{11}=\begin{pmatrix}
        0.7 & 0.1
    \end{pmatrix},~Q_{21}=\begin{pmatrix}
        1 &0 \\ 0 & 1
    \end{pmatrix}, \\
    &Q_{22}=\begin{pmatrix}
        0.4 \\ 0.3
    \end{pmatrix},~R_{22}=\begin{pmatrix}
        0.2 & 0
    \end{pmatrix},~Q_{32}=\begin{pmatrix}
        1 &0 \\ 0 & 1
    \end{pmatrix},\\
    &R_{12}=R_{23}=0.1,~ Q_{33}=\begin{pmatrix}
        0.3 \\ 0.6
    \end{pmatrix},~R_{33}=\begin{pmatrix}
        0 & 0.4
    \end{pmatrix}.
\end{align*}

These coefficients have been chosen to make the whole interconnection unstable in open-loop. However, the corresponding divergence rate is small to avoid numerical issues (having a greater divergence rate would require more accuracy in the simulations and, consequently, a longer simulation time). 
\modif{The output-feedback control law presented in Section~\ref{Chap_3_Sec_Output_Feedback} \modif{is combined with a well-tuned low-pass filter (see~Section~\ref{Sec_filter}). We used a simple low-pass filter of $4\text{th}$ order with a bandwidth of $125~ \text{rad.s}^{-1}$. We added an input delay of $0.1s$ to show the robustness of the design to small delays in the loop. Some parameters are subject to constant multiplicative uncertainties (up to 5\%).}
% \modif{We added a random mismatch (up to $5\%$) on some coefficients. Moreover, we consider an input delay of $0.1s$}. 
We have pictured in Figure~\ref{fig:OL-CL} the time evolution of the $L^2$-norm of the open-loop system and the closed-loop system. As expected, the resulting system exponentially converges to zero. We do not have finite-time convergence due to the filter and the different numerical approximations.} The corresponding control effort has been plotted in Fig~\ref{fig:control_effort}.

 \begin{figure}[ht]
\begin{center}
		\includegraphics[width=0.8\columnwidth]{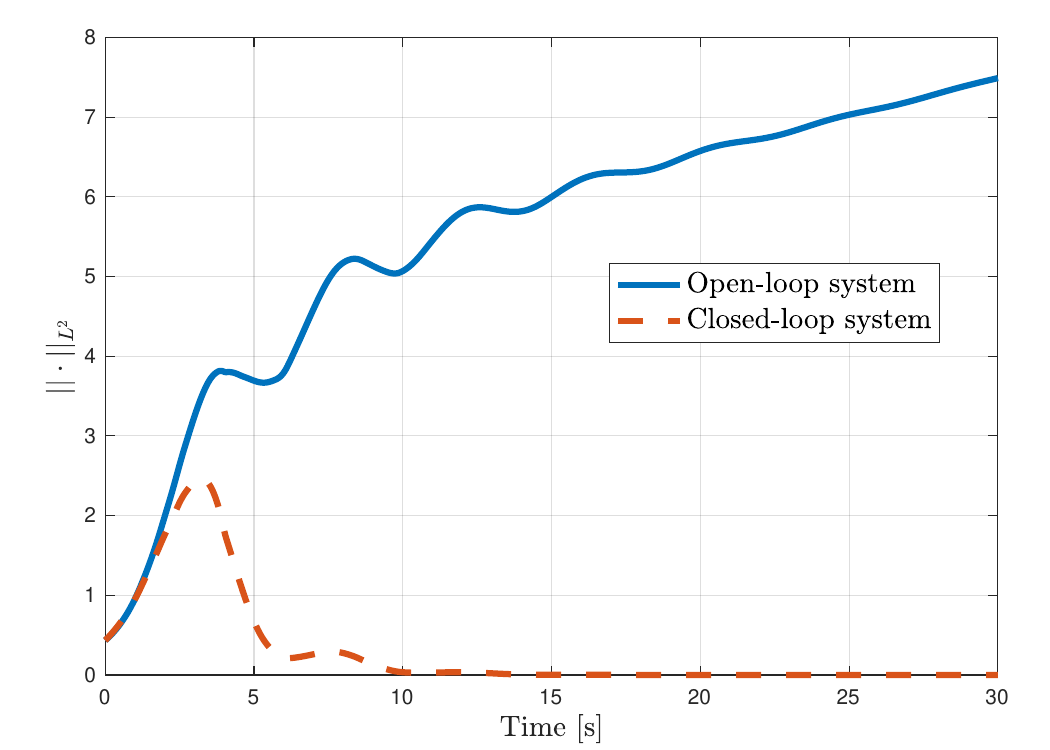}
	\caption{Evolution of the $L^2$-norm of the system~\eqref{Chap_3_eq_init_PDE_u}-\eqref{Chap_3_eq_init_bound_v} in open-loop (blue) and in closed-loop (dashed red), using the (filtered) output-feedback control law presented in Section~\ref{Chap_3_Sec_Output_Feedback}.}\label{fig:OL-CL}
	\end{center}
\end{figure}

 \begin{figure}[ht]
\begin{center}
		\includegraphics[width=0.8\columnwidth]{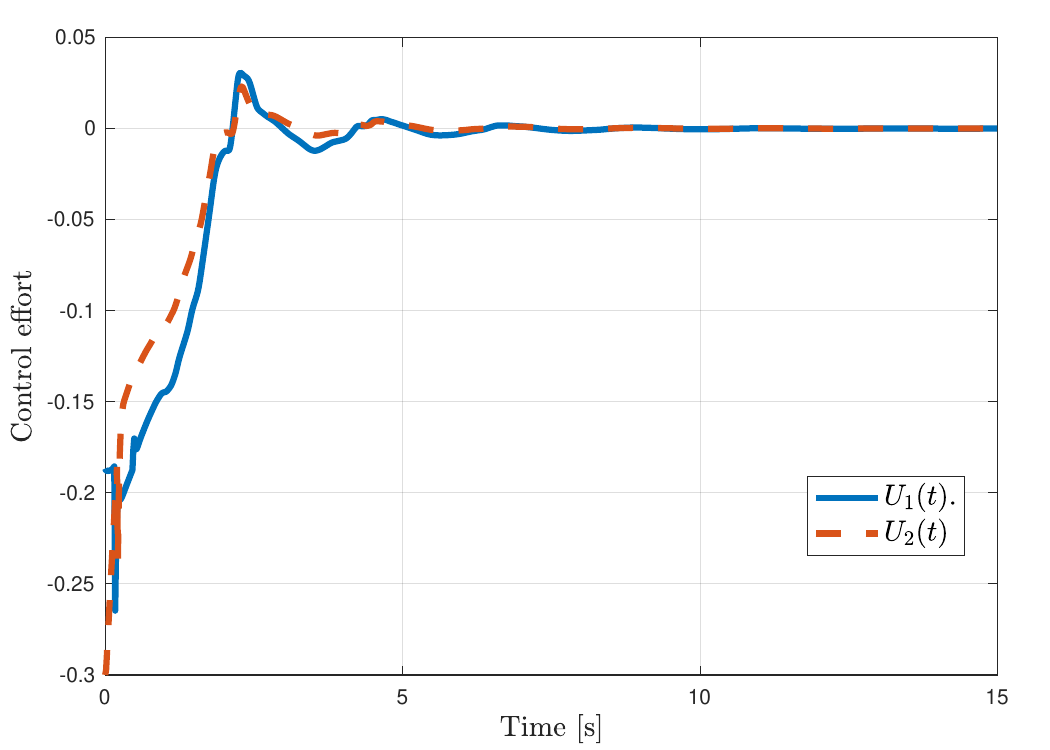}
	\caption{Evolution of the control effort $U_1(t)$ and $U_2(t)$.}\label{fig:control_effort}
	\end{center}
\end{figure}
\section{Conclusion and perspectives} \label{Section_Conclusion}

In this paper, we have introduced a recursive methodology to design a stabilizing output-feedback controller for a network of $N$ PDE subsystems with a chain structure. The different subsystems are interconnected through their boundaries, and the control input is located at one extremity of the chain. The proposed framework required several fundamental properties for each subsystem: output trajectory tracking, input-to-state stability, predictability (we can design predictors of the different states), and observability. We have shown that these properties were always satisfied for hyperbolic subsystems. The proposed approach is modular in that additional subsystems can easily be included. Moreover, we believe the proposed framework can be extended to different types of subsystems (such as ODEs and parabolic equations) as it has been done in~\cite{redaud2021SCL} with an ODE at the end of the chain, provided similar properties can still be verified. Recent results have been developed in~\cite{xu2023stabilization} for parabolic systems using an analogous recursive approach. In~\cite{deutscher2023backstepping}, the authors have considered interconnections between hyperbolic and parabolic systems. It could be interesting in future works to analyze if the control strategy developed in~\cite{deutscher2023backstepping} could be applied to interconnections of hyperbolic systems and compared with our methodology.

\modif{One current limitation of our proposed approach is its high complexity and computational burden. We need to compute state predictions for each subsystem composing the interconnection, which may be time-consuming.} This numerical burden may explode with the number of subsystems, thus making any implementation impossible. To leverage the numerical effort induced by these controllers, it may be necessary to approximate them (e.g., by finite-dimensional systems). This emphasizes the necessity of investigating the questions of model reduction using late-lumping techniques~\cite{ecklebe2017approximation,auriol2019late}. Recently, {machine-learning approximations (based on the DeepONet algorithm) have been successfully tested in \cite{shi2022machine} on simple examples, but there is no general proof of convergence yet. Concerning implementing the proposed recursive control law, we underline that the robustness aspects have been neglected in this paper. However, we believe the results from~\cite{auriol2023robustification} can be adjusted to cover the proposed control strategy. In future works, we will also consider the case of having the actuator located at one of the intersection nodes of the chain. As shown in~\cite{redaud2022stabilizing,redaud2021stabilizing}, this raises challenging \textbf{controllability questions}. In most cases, such interconnected systems may not be controllable, and appropriate controllability conditions must be derived.

\appendices
\section{Technical Lemma} \label{proof_Prop_track}

\begin{lemma} \label{Chap_3_Lemma_Characteristics}
Consider the $i^{th}$ subsystem ($1\leq i \leq N$).
    There exist matrix functions $\check G_i$ and $\tilde G_i$ such that for all $1\leq j \leq n_i$, for all $x \in [0,1]$, for all $t>t_0+\frac{1}{\lambda_i^1}$
\begin{align}
    \alpha^j_i(t,x)&=\alpha^j_i(t+\frac{1-x}{\lambda_i^j},1)\nonumber \\
    &+\sum_{\ell=j+1}^{n_i}\int_0^{\frac{1-x}{\lambda_i^j}}(\check G_i(x,\nu))_{j\ell}\alpha_i^\ell(t+\nu,1)d\nu. \label{Chap_3_eq_alpha_delay_1} \\
    \alpha^j_i(t,x)&=\alpha^j_i(t-\frac{x}{\lambda_i^j},0)\nonumber \\
    &+\sum_{\ell=j+1}^{n_i}\int_0^{\frac{x}{\lambda_i^j}}(\tilde G_i(x,\nu))_{j\ell}\alpha_i^\ell(t-\nu,0)d\nu. \label{Chap_3_eq_alpha_delay}
\end{align}
\end{lemma}
\begin{proof}
We will only prove that equation~\eqref{Chap_3_eq_alpha_delay_1}. The proof of equation~\eqref{Chap_3_eq_alpha_delay} is analogous. The proof will be done recursively.  
     Applying the method of characteristics to equation~\eqref{Chap_3_PDE_alpha}, we obtain for all $1\leq j \leq n_i$, for all $x \in [0,1]$, and for all $t>t_0+\frac{1}{\lambda_i^1}$,
\begin{align}
    &\alpha^j_i(t,x)=\alpha_i^j(t+\frac{1-x}{\lambda_i^j},1)\nonumber \\
    &-\sum_{k=j+1}^{n_i}\int_0^{\frac{1-x}{\lambda_i^j}}(G_i(x+\lambda_i^j\nu))_{jk}\alpha_i^k(t+\nu,x+\lambda_i^j \nu)d\nu. \label{Chap_3_charac_x}
\end{align}
In particular, since the matrix $G_i$ is strictly upper-triangular, we obtain 
\begin{align*}
    \alpha^{n_i}_i(t,x)&=\alpha_{n_i}^j(t+\frac{1-x}{\lambda_i^{n_i}},1).
\end{align*}
Let us now consider $1< j \leq n_i$ and assume that there exist matrix functions $\check G_i$ such that equation~\eqref{Chap_3_eq_alpha_delay_1} holds for any $k\geq j$. Applying equation~\eqref{Chap_3_eq_alpha_delay_1} to equation~\eqref{Chap_3_charac_x}, we obtain
\begin{align*}
    &\alpha^{j-1}_i(t,x)=\alpha_i^{j-1}(t+\frac{1-x}{\lambda_i^{j-1}},1)-\sum_{k=j}^{n_i}\int_0^{\frac{1-x}{\lambda_i^{j-1}}}(G_i(x+\nonumber \\
    &\lambda_i^{j-1}\nu))_{j-1,k}[\alpha_i^k(t+\nu+\frac{1-x-\lambda_i^{j-1}\nu}{\lambda_i^{k}},1) +\sum_{\ell=k+1}^{n_i} \\
      & \int_0^{\frac{1-x-\lambda_i^{j-1}\nu}{\lambda_i^{k}}} (\check G_i(x+\lambda_i^{j-1} \nu,\eta))_{k\ell} \alpha_i^\ell(t+\nu+\eta,1)d\eta]d\nu. 
\end{align*}
We obtain the desired expression by performing changes of variables and applying Fubini's theorem.
 $\blacksquare$
\end{proof}


\bibliographystyle{plain}
\bibliography{biblio}

%\end{thebibliography}


\begin{IEEEbiography}
[{\includegraphics[width=1in,height=1.25in,clip,keepaspectratio]{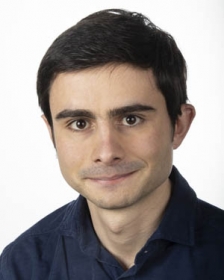}}]{Jean Auriol}
received his Master degree in civil engineering in 2015 in Mines Paris, part of PSL Research University and in 2018
his Ph.D. degree in control theory and applied mathematics from the same university (Centre Automatique
et Systèmes). His Ph.D. thesis, titled Robust design of
backstepping controllers for systems of linear hyperbolic PDEs, has been nominated for the best thesis award given by the GDR MACS and the Section Automatique du Club EEA in France. From 2018 to 2019,
he was a Postdoctoral Researcher at the Department
of Petroleum Engineering, University of Calgary, AB, Canada, where he was
working on the implementation of backstepping control laws for the attenuation of mechanical vibrations in drilling systems. From December 2019, he is a Researcher (Chargé de Recherches) at CNRS, Université Paris-Saclay, Centrale Supelec, Laboratoire des Signaux et Systèmes (L2S), Gif-sur-Yvette, France. His research interests include robust control of hyperbolic systems, neutral systems, and interconnected systems
\end{IEEEbiography}


%\newpage


%\vfill


\end{document}